\documentclass{amsart}
\usepackage{amssymb}
\usepackage{amsmath}
\usepackage{amsthm}
\usepackage{empheq}

\usepackage{float}
\restylefloat{table}

\usepackage[all]{xy}

\theoremstyle{definition}
\newtheorem*{definition}{Definition}

\newtheorem*{remark}{Remark}

\newtheorem{theorem}{Theorem}[section]
\newtheorem{lemma}[theorem]{Lemma}
\newtheorem{proposition}[theorem]{Proposition}

\begin{document}

\title{On Chudnovsky-Ramanujan Type Formulae}

\author{Imin Chen and Gleb Glebov}

\date{September 2016}



\address{Imin Chen \\
Department of Mathematics \\
Simon Fraser University \\
Burnaby \\
British Columbia \\
CANADA}

\email{ichen@sfu.ca}

\address{Gleb Glebov \\
Department of Mathematics \\
Simon Fraser University \\
Burnaby \\
British Columbia \\
CANADA}

\email{gglebov@sfu.ca}

\thanks{This work was supported by an NSERC Discovery Grant and SFU VPR Bridging Grant.}

\maketitle

\begin{abstract}
In a well-known 1914 paper, Ramanujan gave a number of rapidly converging series for $1/\pi$ which are derived using modular functions of higher level. D. V. and G. V. Chudnovsky in their 1988 paper derived an analogous series representing $1/\pi$ using the modular function $J$ of level 1, which results in highly convergent series for $1/\pi$, often used in practice. In this paper, we explain the Chudnovsky method in the context of elliptic curves, modular curves, and the Picard-Fuchs differential equation. In doing so, we also generalize their method to produce formulae which are valid around any singular point of the Picard-Fuchs differential equation. Applying the method to the family of elliptic curves parameterized by the absolute Klein invariant $J$ of level 1, we determine all Chudnovsky-Ramanujan type formulae which are valid around one of the three singular points: $0, 1, \infty$.
\end{abstract}

\tableofcontents

\section{Introduction}

The formula
\begin{equation}
\label{chudnovsky-formula-pi}
12\sum_{n = 0}^\infty (-1)^n \frac{545140134n + 13591409}{(640320^3)^{n + 1/2}} \frac{(6n)!}{(3n)! n!^3} = \frac{1}{\pi}
\end{equation}
has been discovered by D. V. and G. V. Chudnovsky \cite{Chudnovsky} (see also \cite{Chudnovsky2}) and has been used in practice for a record breaking computation of the digits of $\pi$. In 2013, $12.1 \times 10^{12}$ digits of $\pi$ were calculated by A. J. Yee and S. Kondo using the Chudnovsky series for $\pi$. Ramanujan \cite{Ramanujan2} was the first to give examples of such formulae.

The derivation of formulae like \eqref{chudnovsky-formula-pi} has traditionally involved specialized knowledge of identities of classical functions and modular functions (see, for instance, \cite{Borwein}, \cite{Chan}). In particular, the formula \eqref{chudnovsky-formula-pi} is derived by proving the following precursor formula (see Theorem~\ref{thm:infty}) valid for imaginary quadratic $\tau$ in a certain simply-connected domain $C_{1/J, \infty}$ of $\mathbb{C}$,
\begin{equation}
\label{Jinftypre}
\frac{a}{\pi \sqrt{d}} \frac{\sqrt{J}}{\sqrt{J - 1}} = F^2 \frac{1 - s_2(\tau)}{6} - J \frac{d}{dJ} F^2,
\end{equation}
where $F = {}_{2}F_{1}(1/12, 5/12; 1; 1/J)$, $J = J(\tau)$, and
\begin{equation}
\label{Chudnovsky}
\tau = \frac{-b + \sqrt{-d}}{2a}, \quad a > 0, \quad -d = b^2 - 4ac
\end{equation}
such that $a$, $b$, $c$ are integers. For $\tau$ in the upper-half plane, $J(\tau) = j(\tau)/12^3$ is the absolute Klein invariant, $j(\tau)$ is the classical $j$-invariant, and $s_2(\tau)$ is defined by
$$s_2(\tau) = \frac{E_4(\tau)}{E_6(\tau)} \bigg(E_2(\tau) - \frac{3}{\pi {\rm Im}(\tau)}\bigg).$$
Here, $E_2(\tau)$, $E_4(\tau)$, $E_6(\tau)$ are the normalized Eisenstein series (see Section 2.4). It is known that $j(\tau)$ is rational for $\tau$ corresponding to imaginary quadratic orders of class number one. Evaluation at such $\tau$ and using Clausen's identity leads to \eqref{chudnovsky-formula-pi}.

The purpose of this paper is to explain the Chudnovsky method \cite{Chudnovsky} in detail and in a way that links it to the modern theory of elliptic and modular curves. Secondly, we generalize the method to work for monodromy around any elliptic point, in addition to the cusps. In doing so, we give a fairly transparent but broad framework which in principle suggests a systematic way of tabulating such formulae for genus zero congruence groups commensurable with ${\rm SL}_2(\mathbb{Z})$, or at least those which are triangular. A feature of this method is that it no longer requires specialized knowledge of classical functions and identities, but rather an explicit form of the Picard-Fuchs differential equation and Kummer's method \cite{Kummer} to determine its hypergeometric solutions \cite{Archinard}.

Even in the case of $J$ of level 1, the method gives new formulae (no longer for $1/\pi$, but related constants), corresponding to monodromy around $J = 0$ and $J = 1$, which are derived in this paper. In particular, we prove the following Chudnovsky-Ramanujan type formulae:

\begin{theorem}
\label{thm:one}
If $\tau$ is as in \eqref{Chudnovsky} and lies in $C_{(J - 1)/J, i}$, then
$$\frac{\tau + i}{2\pi \alpha^2 \sqrt{3}} \frac{\sqrt{J}}{\sqrt{1 - J}} \bigg(a\frac{\tau + i}{\sqrt{-d}} - 1\bigg) = F^2 \frac{1 - s_2(\tau)}{6} - J\frac{d}{dJ} F^2,$$
where $F = {}_{2}F_{1}(1/12, 5/12; 1/2; (J - 1)/J)$, $J = J(\tau)$, $\alpha = 2i \eta(i)^2$, and the principal branch of the square root is used. Here, $\eta(\tau)$ is the Dedekind eta function.
\end{theorem}

\begin{theorem}
\label{thm:zero}
If $\tau$ is as in \eqref{Chudnovsky} and lies in $C_{J/(J - 1), \rho}$, then
$$-\frac{\tau - \overline{\rho}}{2\pi \alpha^2 \sqrt{3}} \frac{J^{1/3}}{(1 - J)^{1/3}} \bigg(a\frac{\tau - \overline{\rho}}{\sqrt{-d}} - 1\bigg) = F^2 \bigg[\frac{J}{6(1 - J)} + \frac{s_2(\tau)}{6}\bigg] + J\frac{d}{dJ} F^2,$$
where $F = {}_{2}F_{1}(1/12, 7/12; 2/3; J/(J - 1))$, $J = J(\tau)$, $\alpha = i \eta(\rho)^2 \sqrt{3}$, and the principal branch of the square and cube root is used. Here, $\eta(\tau)$ is the Dedekind eta function and $\rho = e^{2\pi i/3}$ is the cube root of unity.
\end{theorem}

The simply-connected domains $C_{1/J, \infty}$, $C_{(J - 1)/J, i}$, $C_{J/(J - 1), \rho}$ are defined and depicted in Section \ref{connected-defn}. In Section \ref{one:examples} and \ref{zero:examples}, we list all possible identities obtained by evaluating the above formulae at $\tau$ corresponding to an imaginary quadratic order of class number one.

\begin{remark}
It is known that
$$\eta(i) = \frac{\Gamma(1/4)}{2\pi^{3/4}},$$
where $\Gamma(z)$ is the gamma function \cite[p. 481, 517]{Watson}. However, the value of $\eta(\rho)$ is not as well-known as $\eta(i)$. We prove in Section \ref{zero:examples} that
$$\eta(\rho)^2 = \frac{3^{1/4} \Gamma(1/3)^3}{4\pi^2 e^{\pi i/12}}.$$
\end{remark}

\section{Preliminaries}

In this section, we recall some basic definitions and facts we need later in the paper.

\subsection{Singular values}

It is known from the theory of elliptic curves with complex multiplication that $j(\tau)$ and $J(\tau)$ are rational for $\tau = \sqrt{-N}$ if $N = 1, 2, 3, 4, 7$ and also for $\tau = \frac{-1 + \sqrt{-N}}{2}$ if $N = 3, 7, 11, 19, 27, 43, 67, 163$ \cite[pp. 237-238]{Cox}. Below are tables giving these rational values.

\begin{table}[H]
\centering
\begin{tabular}{ccc}
\hline
$N$ & $j(\tau)$ & $J(\tau)$ \\
\hline
$1$ & $12^3$ & $1$ \\
$2$ & $20^3$ & $5^3/3^3$ \\
$3$ & $2 \cdot 30^3$ & $5^3/2^2$ \\
$4$ & $66^3$ & $11^3/2^3$ \\
$7$ & $255^3$ & $5^3 17^3/2^6$ \\
\hline
\end{tabular}
\caption{Special values of $j(\tau)$ and $J(\tau)$ at $\tau = \sqrt{-N}$}
\end{table}

\begin{table}[H]
\centering
\begin{tabular}{ccc}
\hline
$N$ & $j(\tau)$ & $J(\tau)$ \\
\hline
$3$ & $0$ & $0$ \\
$7$ & $-15^3$ & $-5^3/4^3$ \\
$11$ & $-32^3$ & $-8^3/3^3$ \\
$19$ & $-96^3$ & $-8^3$ \\
$27$ & $-3 \cdot 160^3$ & $-40^3/3^2$ \\
$43$ & $-960^3$ & $-80^3$ \\
$67$ & $-5280^3$ & $-440^3$ \\
$163$ & $-640320^3$ & $-53360^3$ \\
\hline
\end{tabular}
\caption{Special values of $j(\tau)$ and $J(\tau)$ at $\tau = \frac{-1 + \sqrt{-N}}{2}$}
\end{table}

\subsection{The hypergeometric function}

The Pochhammer symbol is defined by
$$(\alpha)_n = \alpha(\alpha + 1)(\alpha + 2) \cdots (\alpha + n - 1),$$
where $n$ is a positive integer and $(\alpha)_0 = 1$, and it is easy to show that $(\alpha)_n = \Gamma(\alpha + n)/\Gamma(\alpha)$, where $\Gamma(z)$ is the gamma function.

\begin{definition}
Let $z \in \mathbb{C}$. The hypergeometric function is defined by
$${}_{2}F_{1}(a, b; c; z) = \sum_{n = 0}^\infty \frac{(a)_n (b)_n}{(c)_n} \frac{z^n}{n!},$$
where $a$, $b$, $c$ are constants.
\end{definition}

\begin{definition}
Let $z \in \mathbb{C}$. The generalized hypergeometric function is defined by
$${}_{3}F_{2}(a_1, a_2, a_3; b_1, b_2; z) = \sum_{n = 0}^\infty \frac{(a_1)_n (a_2)_n (a_3)_n}{(b_1)_n (b_2)_n} \frac{z^n}{n!},$$
where $a_1$, $a_2$, $a_3$, $b_1$, $b_2$ are constants.
\end{definition}

It is easy to show that both ${}_{2}F_{1}(a, b; c; z)$ and ${}_{3}F_{2}(a_1, a_2, a_3; b_1, b_2; z)$ converge absolutely if $|z| < 1$.

Now, since ${}_{2}F_{1}(a, b; c; z)$ is a power series, it can be differentiated termwise within its radius of convergence. Hence it is seen that
\begin{equation}
\label{hyperdiff}
\frac{d}{dz} {}_{2}F_{1}(a, b; c; z) = \frac{ab}{c} {}_{2}F_{1}(a + 1, b + 1; c + 1; z).
\end{equation}

The hypergeometric function satisfies the hypergeometric differential equation:
$$z(1 - z) \frac{d^2 y}{dz^2} + [c - (a + b + 1)z] \frac{dy}{dz} - aby = 0,$$
where $a$, $b$, $c$ are constants. In particular, when $c$, $c - a - b$, $a - b$ are not integers, we find the following pairs of fundamental solutions \cite{Kummer}.

Around $z = 0$:
\begin{align*}
&{}_{2}F_{1}(a, b; c; z), \\
&z^{1 - c} {}_{2}F_{1}(a - c + 1, b - c + 1; 2 - c; z).
\end{align*}

Around $z = 1$:
\begin{align*}
&{}_{2}F_{1}(a, b; a + b - c + 1; 1 - z), \\
&(1 - z)^{c - a - b} {}_{2}F_{1}(c - a, c - b; c - a - b + 1; 1 - z).
\end{align*}

Around $z = \infty$:
\begin{align*}
&z^{-a} {}_{2}F_{1}(a, a - c + 1; a - b + 1; z^{-1}), \\
&z^{-b} {}_{2}F_{1}(b, b - c + 1; b - a + 1; z^{-1}).
\end{align*}

\subsection{Invariants}
\label{identities}

Let $E$ be an elliptic curve over $\mathbb{C}$ given by
$$E : y^2 = 4x^3 - g_2 x - g_3, \quad g_2, g_3 \in \mathbb{C}, \quad g_2^3 - 27g_3^2 \ne 0.$$
The quantity $\Delta = \Delta(E) = g_2^3 - 27g_3^2$ is called the (normalized) discriminant of $E$. The $j$-invariant of $E$ is defined by
$$j = j(E) = 12^3 \frac{g_2^3}{\Delta},$$
and its absolute Klein invariant $J$ is defined by
$$J = J(E) = \frac{j}{12^3}.$$

\begin{lemma}
\label{trivial}
We have
$$\frac{3g_3}{2g_2} \frac{\sqrt{J}}{\sqrt{J - 1}} = \frac{(J\Delta)^{1/6}}{\sqrt{12}}.$$
\end{lemma}

\begin{proof}
Up to a 6-th root of unity, we have that
\begin{align*}
\frac{3g_3}{2g_2} (j\Delta)^{-1/6} \frac{\sqrt{J}}{\sqrt{J - 1}} &= \frac{3g_3}{2g_2} (12^3 g_2^3)^{-1/6} \sqrt{\frac{12^3 g_2^3/\Delta}{12^3 g_2^3/\Delta - 12^3}} \\
&= \frac{3g_3}{2g_2} \frac{12g_2}{\sqrt{(12^3 g_2^3/\Delta - 12^3) \Delta}} \\
&= \frac{3g_3}{4\sqrt{3g_2^3 - 3\Delta}}.
\end{align*}
But $\Delta = g_2^3 - 27g_3^2$, so
$$\frac{3g_3}{4\sqrt{3g_2^3 - 3\Delta}} = \frac{3g_3}{4\sqrt{81 g_3^2}},$$
which simplifies to 1/12.
\end{proof}

In what follows, we will often not specify the branch of the $n$-th root function during intermediate calculations. This means that any formulae stated will only be valid up to some $n$-th root of unity (such as in the lemma above), where $n \le 6$.

\begin{remark}
In the main theorems of this paper (Theorem \ref{thm:one}, \ref{thm:zero}, and \ref{thm:infty}), we in fact have exact equalities using the principal branch of the $n$-th root function. This is established by the following technique: Suppose we have the identity $f(\tau) = \epsilon(\tau)g(\tau)$ on some connected subset $U$ of $\mathbb{C}$, where $f(\tau)$ and $g(\tau)$ are continuous non-vanishing functions on $U$, $\epsilon(\tau)$ is an $n$-th root of unity, and $n$ does not depend on $\tau \in U$. Then $f(\tau)/g(\tau)$ is a continuous function on a connected set $U$ to a finite set of $n$-th roots of unity, which has the discrete topology under the induced subspace topology inherited from $\mathbb{C}$. Thus, if there is one $\tau_0 \in U$ such that $f(\tau_0)/g(\tau_0) = 1$, then in fact $\epsilon(\tau) = 1$ on all of $U$.
\end{remark}

\subsection{Eisenstein series}
Let $\Lambda \subset \mathbb{C}$ be a lattice. For an integer $k > 1$, the Eisenstein series $G_{2k}(\Lambda)$ attached to the lattice $\Lambda$ are defined as
$$G_{2k}(\Lambda) = \sum_{\omega \in \Lambda'} \frac{1}{\omega^{2k}}, \quad \Lambda' = \Lambda \setminus \{0\}.$$
Moreover, we define $g_2(\Lambda)$ and $g_3(\Lambda)$ by
$$g_2(\Lambda) = 60G_4(\Lambda) \quad \text{and} \quad g_3(\Lambda) = 140G_6(\Lambda),$$
and the discriminant function $\Delta(\Lambda)$ for the lattice $\Lambda$ by $\Delta(\Lambda) = g_2(\Lambda)^3 - 27g_3(\Lambda)^2$.

If the lattice $\Lambda$ has an ordered $\mathbb{Z}$-basis $(\omega_1, \omega_2)$, where ${\rm Im}(\omega_2/\omega_1) > 0$, we may define $G_{2k}(\Lambda) = G_{2k}(\omega_1, \omega_2)$. For any $\alpha \ne 0$ we have that
$$G_{2k}(\alpha \Lambda) = \alpha^{-2k} G_{2k}(\Lambda).$$
Thus, if $\alpha = 1/\omega_1$ and $\tau = \omega_2/\omega_1$, we obtain
$$G_{2k}(1, \tau) = G_{2k}(\tau) = \omega_1^{2k} G_{2k}(\Lambda),$$
that is,
$$G_{2k}(\tau) = \sum_{(m, n) \ne (0, 0)} \frac{1}{(m + n\tau)^{2k}},$$
where the summation is over all integers $m$, $n$ that do not vanish simultaneously, and $G_{2k}(\tau)$ converges absolutely if $k \geq 2$ and ${\rm Im}(\tau) > 0$.

Let $q = e^{2\pi i \tau}$ with ${\rm Im}(\tau) > 0$, and recall that for an integer $k \geq 2$, the normalized Eisenstein series are defined by
$$E_{2k}(\tau) = \frac{G_{2k}(\tau)}{2\zeta(2k)}.$$
Moreover, it is well-known that if $k \geq 2$ and ${\rm Im}(\tau) > 0$, then
$$E_{2k}(\tau) = 1 + \frac{(2\pi i)^{2k}}{\zeta(2k) (2k - 1)!} \sum_{n = 1}^\infty n^{2k - 1} \frac{q^n}{1 - q^n}.$$
Although $G_2(\tau)$ is conditionally convergent, we can still define $E_2(\tau)$ by the formula above because it is valid for $k = 1$. In particular,
\begin{align*}
E_2(\tau) &= 1 - 24 \sum_{n = 1}^\infty n \frac{q^n}{1 - q^n}, \\
E_4(\tau) &= 1 + 240 \sum_{n = 1}^\infty n^3 \frac{q^n}{1 - q^n}, \\
E_6(\tau) &= 1 - 504 \sum_{n = 1}^\infty n^5 \frac{q^n}{1 - q^n}.
\end{align*}

Furthermore, using Ramanujan's differential equations \cite[p. 142]{Ramanujan} (see also \cite{Ramanujan3})
\begin{equation}
\label{theta-op}
q\frac{d}{dq} E_4(\tau) = \frac{E_2(\tau) E_4(\tau) - E_6(\tau)}{3} \quad \text{and} \quad q\frac{d}{dq} E_6(\tau) = \frac{E_2(\tau) E_6(\tau) - E_4(\tau)^2}{2},
\end{equation}
and the identity,
$$J(\tau) = \frac{E_4(\tau)^3}{E_4(\tau)^3 - E_6(\tau)^2},$$
we obtain
\begin{equation}
\label{ramanujan-diff}
\frac{J'}{J} = -2\pi i \frac{E_6(\tau)}{E_4(\tau)},
\end{equation}
where $J' = \frac{dJ}{d\tau}$.

\begin{remark}
The equations in \eqref{theta-op} can be viewed as determining the effect of Ramanujan's $\theta$-operator on the ring of modular forms of level 1 \cite{Serre}.
\end{remark}

\subsection{Uniformization}

For $z \in \mathbb{C}$, the Weierstrass function $\wp(z, \Lambda)$ is defined by
$$\wp(z, \Lambda) = \frac{1}{z^2} + \sum_{\omega \in \Lambda'} \bigg[\frac{1}{(z - \omega)^2} - \frac{1}{\omega^2}\bigg],$$
where $\Lambda \subset \mathbb{C}$ is a lattice and $\Lambda' = \Lambda \setminus \{0\}$.

Consider the elliptic curve $E_{\Lambda}$ over $\mathbb{C}$ given in Weierstrass form by
$$E_{\Lambda} : y^2 = 4x^3 - g_2(\Lambda) x - g_3(\Lambda).$$
The map $\iota_\Lambda : \mathbb{C} / \Lambda \xrightarrow{\cong} E(\mathbb{C})$, with $\iota_\Lambda : z \mapsto [\wp(z, \Lambda) : \wp'(z, \Lambda) : 1]$, is an isomorphism of Riemann surfaces \cite[Chapter VI, Section 5, Proposition 5.2]{Silverman}. The uniformization theorem \cite[Chapter VI, Section 5, Theorem 5.1]{Silverman} states that every elliptic curve $E$ over $\mathbb{C}$ arises in this way, that is, given any $g_2, g_3 \in \mathbb{C}$ satisfying $g_2^3 - 27g_3^2 \ne 0$, there exists a unique lattice $\Lambda \subset \mathbb{C}$ such that $g_2 = g_2(\Lambda)$ and $g_3 = g_3(\Lambda)$, and hence $E(\mathbb{C}) = E_{\Lambda}(\mathbb{C})$.

Given an ordered $\mathbb{Z}$-basis $(\omega_1, \omega_2)$ for $\Lambda$ with ${\rm Im}(\omega_2/\omega_1) > 0$, which is not unique, the quantities $\omega_1$, $\omega_2$ are called the periods of $E$. Note that a $\mathbb{Z}$-basis should be specified before we can talk about the periods of $E$ as a pair of complex numbers. We will specify the basis used in the proofs of the main theorems in Section \ref{Periods}.

\section{Periods and families of elliptic curves}

\subsection{Families of elliptic curves}

For $u \ne 0$, the map $\varphi_u : (x, y) \mapsto (u^2 x, u^3 y)$ gives an isomorphism $\varphi_u : E \xrightarrow{\cong} E'$ from an elliptic curve $E$ over $\mathbb{C}$ to the elliptic curve $E'$ over $\mathbb{C}$, where
$$E : y^2 = 4x^3 - g_2 x - g_3 \quad \text{and} \quad E' : y^2 = 4x^3 - g_2' x - g_3',$$
and
\begin{align*}
g_2' &= u^4 g_2, \\
g_3' &= u^6 g_3, \\
\Delta(E') &= u^{12} \Delta(E).
\end{align*}

By the uniformization theorem, $E(\mathbb{C}) = E_{\Lambda}(\mathbb{C})$ for some lattice $\Lambda \subset \mathbb{C}$. It then follows that $E'(\mathbb{C}) = E_{u^{-1} \Lambda}(\mathbb{C})$, and we have the following commutative diagram:
\begin{displaymath}
\xymatrix{
\mathbb{C} / \Lambda
\ar[d]_{z \mapsto u^{-1} z}
\ar[r]_{\iota_\Lambda}
& E_\Lambda(\mathbb{C}) = E(\mathbb{C})
\ar[d]^{\varphi_{u}} \\
\mathbb{C} / {u^{-1} \Lambda}
\ar[r]_{\iota_{u^{-1} \Lambda}}
& E_{u^{-1} \Lambda}(\mathbb{C}) = E'(\mathbb{C})
}
\end{displaymath}
so that the isomorphism $\varphi_u$ corresponds to scaling $\Lambda$ by $u^{-1}$.

We now introduce four families of elliptic curves: $E$, $E_\tau$, $\widetilde{E}$, $E_J$, and compare their discriminants, associated lattices, and periods.

Consider the elliptic curve $E$ over $\mathbb{C}$ given by
$$E : y^2 = 4x^3 - g_2 x - g_3, \quad \Delta(E) = \Delta = g_2^3 - 27g_3^2, \quad \Lambda(E) = \mathbb{Z} \omega_1 + \mathbb{Z} \omega_2,$$
with discriminant $\Delta(E)$ and associated lattice $\Lambda(E)$.

Taking $u = \omega_1$, we see that $E$ is isomorphic to
$$E_\tau : y^2 = 4x^3 - g_2(\tau) x - g_3(\tau), \quad \Delta(E_\tau) = \Delta(\tau) = g_2(\tau)^3 - 27g_3(\tau)^2, \quad \Lambda_\tau = \mathbb{Z} + \mathbb{Z} \tau,$$
with discriminant $\Delta(E_\tau)$ and associated lattice $\Lambda(E_\tau)$, where
$$\tau = \frac{\omega_2}{\omega_1}, \quad \omega_1^4 g_2 = g_2(\tau), \quad \omega_1^6 g_3 = g_3(\tau), \quad \omega_1^{12} \Delta = \Delta(\tau).$$

Taking $u = \Delta^{-1/12}$, we see that $E$ is isomorphic to
$$\widetilde{E} : y^2 = 4x^3 - \gamma_2 x - \gamma_3, \quad \Delta(\widetilde{E}) = 1, \quad \Lambda(\widetilde{E}) = \mathbb{Z} \widetilde{\omega}_1 + \mathbb{Z} \widetilde{\omega}_2,$$
with discriminant $\Delta(\widetilde{E})$ and associated lattice $\Lambda(\widetilde{E})$, where
$$\gamma_2 = \Delta^{-1/3} g_2 = J^{1/3}, \quad \gamma_3 = \Delta^{-1/2} g_3 = \sqrt{\frac{J - 1}{27}}, \quad \widetilde{\omega}_k = \omega_k \Delta^{1/12}.$$

Taking $u = \sqrt{g_2/g_3}$, we see that $E$ is isomorphic to
$$E_J : y^2 = 4x^3 + g(x + 1), \quad \Delta(E_J) = \Delta(J) = \frac{3^9 J^2}{16(1 - J)^2}, \quad \Lambda(E_J) = \mathbb{Z} \Omega_1 + \mathbb{Z} \Omega_2,$$
with discriminant $\Delta(E_J)$ and associated lattice $\Lambda(E_J)$, where
$$g = -\frac{g_2^3}{g_3^2} = \frac{27J}{1 - J}, \quad \Omega_k = \omega_k \sqrt{\frac{g_3}{g_2}}.$$
Note also that $\widetilde{E} \cong E_J$ if
$$u = \Delta^{1/12} \sqrt{\frac{g_2}{g_3}} = J^{1/6} \bigg(\frac{27}{J - 1}\bigg)^{1/4},$$
and $E_\tau$ is isomorphic to $E_J$ if $u = \sqrt{g_2(\tau)/g_3(\tau)}$.

In the above treatment, we related the discriminant, lattice, and periods associated with the elliptic curves $E$, $E_\tau$, $\widetilde{E}$, $E_J$ starting with the elliptic curve $E$. We note however that for the proofs of the main theorems, we in fact begin with the elliptic curve $E_J$, and make a choice of periods $(\Omega_1, \Omega_2) = (\Omega_1(J), \Omega_2(J))$ for $E_J$ first in Section \ref{Periods}. This then fixes the periods for the elliptic curves $E$, $E_\tau$, $\widetilde{E}$.

\begin{remark}
The quantities $\Delta$, $\Delta(\tau)$, $\Delta(J)$ are distinct: $\Delta = \Delta(E)$ is the discriminant of the elliptic curve $E$, $\Delta(\tau) = \Delta(E_\tau)$ is the discriminant of the elliptic curve $E_\tau$, which coincides with the well-known cusp form of weight 12 for $\rm{SL}_2(\mathbb{Z})$ with the same name, and finally $\Delta(J) = \Delta(E_J)$ is the discriminant of the elliptic curve $E_J$. Likewise, $g_2$, $g_3$ and $g_2(\tau)$, $g_3(\tau)$ are distinct.
\end{remark}

\subsection{Periods}
\label{Periods}

It is known that the periods $\Omega_1$, $\Omega_2$ of $E_J$ satisfy the Picard-Fuchs differential equation \cite[p. 34]{Modulfunctionen}
\begin{equation}
\label{eq:Fuchs}
\frac{d^2 \Omega}{dJ^2} + \frac{1}{J} \frac{d\Omega}{dJ} + \frac{31J - 4}{144J^2 (1 - J)^2} \Omega = 0.
\end{equation}
For completeness, we describe below more precisely what this means and how these periods are defined as integrals.

The set of points $E(\mathbb{C})$ is a complex Lie group which is topologically a torus. Let $H_1(E(\mathbb{C}), \mathbb{Z})$ denote the first homology group of the topological space $E(\mathbb{C})$. It is known that $H_1(E(\mathbb{C}), \mathbb{Z}) \cong \mathbb{Z}^2$, where simple loops $\alpha$ and $\beta$ in $E(\mathbb{C})$ can be taken as a $\mathbb{Z}$-basis as depicted in \cite[Chapter VI, Section 1, Figure 6.5]{Silverman}.

Let $X(2)(\mathbb{C})$ denote the modular curve of level 2 corresponding to the Legendre family of elliptic curves $y^2 = x(x - 1)(x - \lambda)$ wtih $\lambda \ne 0, 1$. The natural $S_3$-covering $X(2)(\mathbb{C}) \cong \mathbb{P}^1(\mathbb{C}) \to X(1)(\mathbb{C}) \cong \mathbb{P}^1(\mathbb{C})$ is unramified outside of $\{0, 1, \infty\}$.

To define $(\Omega_1, \Omega_2) = (\Omega_1(J), \Omega_2(J))$ as continuous functions of $J$, it suffices, as per Silverman \cite[Chapter VI, Section 5, Proposition 5.2]{Silverman}, to define a $\mathbb{Z}$-basis $\{\gamma_1(J), \gamma_2(J)\}$ for $H_1(E_J, \mathbb{Z})$, where $\gamma_1(J)$ and $\gamma_2(J)$ are continuously functions of $J$ in some open subset of $X(1)(\mathbb{C}) \setminus \{0, 1, \infty\}$. In particular, we see that
$$\Omega_k = \Omega_k(J) = \int_{\gamma_k(J)} \frac{dx}{y}$$
gives a choice of periods for $E_J$ that are continuous as functions of $J$ in some open subset of $X(1)(\mathbb{C}) \setminus \{0, 1, \infty\}$.

Now, let $U$ be a simply-connected open subset of $X(1)(\mathbb{C}) \setminus \{0, 1, \infty\}$ and let $J_0 \in U$. Let $\sigma$ be a path in $U$ such that $\sigma(0) = J_0$ and $\sigma(1) = J$. By the path-lifting lemma, there exists a unique path $\widetilde{\sigma}$ in $X(2)(\mathbb{C})$ lifting $\sigma$, where $\widetilde{\sigma}(0)$ corresponds to a fixed initial choice of labeling of the distinct roots $e_0(J_0)$, $e_1(J_0)$, $e_\lambda(J_0)$ of the cubic $4x^3 + g(J_0)(x + 1)$. Hence, for $J \in U$, we have a well-defined labeling of the distinct roots $e_0(J)$, $e_1(J)$, $e_\lambda(J)$ of the cubic $4x^3 + g(J)(x + 1)$, which vary continuously for $J \in U$.

Let $\alpha_1(J)$ be a simple loop which encircles $e_0(J)$ and $e_1(J)$, and does not pass through the other root. Let $\alpha_2(J)$ be a simple loop which encircles $e_1(J)$ and $e_\lambda(J)$, and does not pass through the other root. Then $\{\alpha_1(J), \alpha_2(J)\}$ forms a $\mathbb{Z}$-basis for $H_1(E_J, \mathbb{Z})$ as depicted in \cite[Chapter VI, Section 1, Figure 6.5]{Silverman}, where $\alpha_1(J)$ and $\alpha_2(J)$ vary continuously for $J \in U$.

For later purposes, we can also modify $\alpha_1(J)$ and $\alpha_2(J)$ in the following way so that we can apply Theorem ~\ref{thm:Jinfty}, \ref{thm:Jone}, \ref{thm:Jzero}. Let $z_0 = \alpha_2(J_0)/\alpha_1(J_0)$ and $q_0 \in \rm{SL}_2(\mathbb{Z})$ be such that $q_0(z_0)$ lies in the standard fundamental domain
$$\mathcal{F} = \{\tau \in \mathfrak{H} \, : \, |{\rm Re}(\tau)| < 1/2, \, |\tau| > 1\}$$
for $\rm{SL}_2(\mathbb{Z})$, and define
$$\begin{pmatrix}
\gamma_2(J) \\ \gamma_1(J)
\end{pmatrix}
=
q_0
\begin{pmatrix}
\alpha_2(J) \\ \alpha_1(J)
\end{pmatrix}$$
for all $J \in U$. This modification ensures our choice of $(\Omega_1, \Omega_2) = (\Omega_1(J), \Omega_2(J))$ has the property that $\Omega_2/\Omega_1$ lies in the closure of $\mathcal{F}$.

For the proofs of Theorem \ref{thm:one}, \ref{thm:zero}, and \ref{thm:infty}, we take $z_0$ to be near $\infty$, $i$, and $\rho$, respectively, and in $\mathcal{F}$, and $U = \mathcal{F}$, to specify the periods $\Omega_1$ and $\Omega_2$ of $E_J$ as described above.

\subsection{Period expressions}
\label{connected-defn}

We begin by proving the following lemma about connected components.

\begin{lemma}
Let $\mathcal{F}$ be the standard fundamental domain for ${\rm SL}_2(\mathbb{Z})$ given by $\mathcal{F} = \{\tau \in \mathfrak{H} \, : \, |{\rm Re}(\tau)| < 1/2, \, |\tau| > 1\}$. Let $C_{\nu(J)} = \{\tau \in \mathcal{F} \, : \, |\nu(J)| < 1\}$, where $\nu(J)$ is one of $1/J$, $(J - 1)/J$, $J/(J - 1)$. Then $C_{\nu(J)}$ is an open set in $\mathfrak{H}$ which is a union of at most two simply-connected components.
\end{lemma}

\begin{proof}
Let $\mathfrak{H}^* = \mathfrak{H} \cup \mathbb{P}^1(\mathbb{Q)}$. The function $J$ gives a complex analytic isomorphism of Riemann surfaces $J : {\rm SL}_2(\mathbb{Z}) \setminus \mathfrak{H}^* \to \mathbb{P}^1(\mathbb{C})$, a fortiori, a homeomorphism of topological spaces. This still holds for $\nu(J)$ as it is a M\"{o}bius transformation. Let $\partial \mathcal{F}$ denote the boundary of $\mathcal{F}$. The restriction of $\nu(J)$ to $C_{\nu(J)}$ gives a homeomorphism between $C_{\nu(J)}$ and $\{\nu(J) \in \mathbb{P}^1(\mathbb{C}) \, : \, |\nu(J)| < 1\} \setminus \nu(\partial \mathcal{F})$, which is a union of at most two simply-connected components, which implies the same for $C_{\nu(J)}$.
\end{proof}

Let $C_{\nu(J), z}$ denote the connected component of $C_{\nu(J)}$ with $z$ lying in the closure of $C_{\nu(J)}$. The connected components $C_{1/J, \infty}$, $C_{(J - 1)/J, i}$, $C_{J/(J - 1), \rho}$ are depicted in the following plots, where $x = {\rm Re}(\tau)$ and $y = {\rm Im}(\tau)$. Note that the regions are open, unbounded, and extend vertically towards $\infty$. Here, $\rho = e^{2\pi i/3}$ is the cube root of unity.

\begin{figure}[H]
\centering
\includegraphics[scale = 0.7]{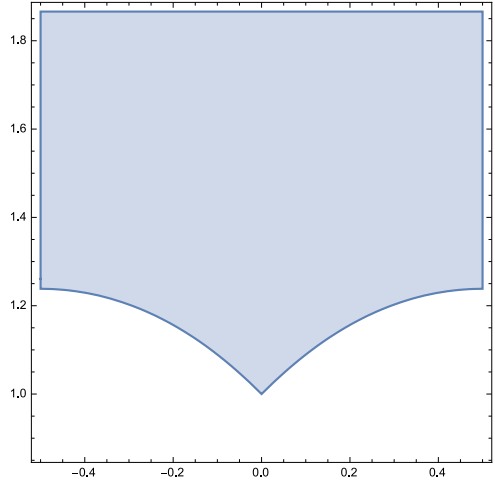}
\caption{$C_{1/J, \infty} = \{\tau \in \mathcal{F} \, : \, |1/J| < 1, \, x^2 + y^2 > 1\}$}
\label{fig:cc1}
\end{figure}

\begin{figure}[H]
\centering
\includegraphics[scale = 0.7]{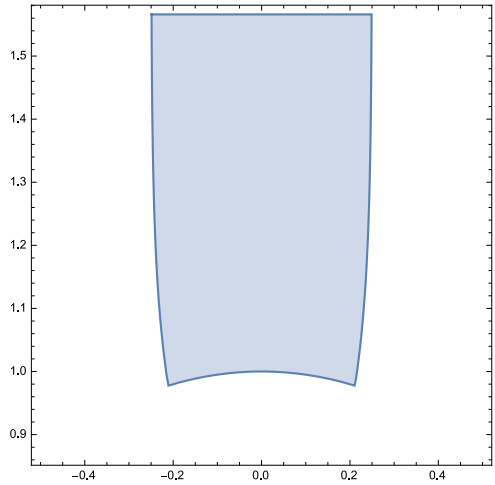}
\caption{$C_{(J - 1)/J, i} = \{\tau \in \mathcal{F} \, : \, |(J - 1)/J| < 1, \, x^2 + y^2 > 1\}$}
\label{fig:cc2}
\end{figure}

\begin{figure}[H]
\centering
\includegraphics[scale = 0.7]{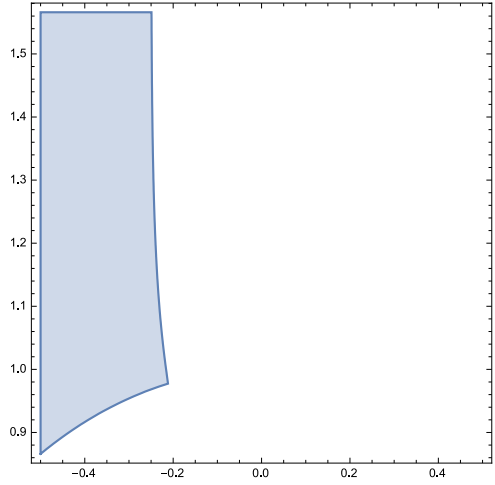}
\caption{$C_{J/(J - 1), \rho} = \{\tau \in \mathcal{F} \, : \, |J/(J - 1)| < 1, \, x^2 + y^2 > 1, \, x < 0\}$}
\label{fig:cc3}
\end{figure}

We now recall some results from \cite{Archinard} which are obtained by Kummer's method of determining the hypergeometric solutions of the Picard-Fuch differential equation \eqref{eq:Fuchs}. Each such solution results in a formula for the period $\Omega_1$ in terms of hypergeometric functions.

\begin{theorem}[$J = \infty$ case]
\label{thm:Jinfty}
Suppose $\tau = \omega_2/\omega_1$ is in the connected component $C_{1/J, \infty}$ of the open set $|J| > 1$. Then
$$\widetilde{\omega}_1 = \omega_1 \Delta^{1/12} = \frac{2\pi}{\sqrt{2\sqrt{3}}} J^{-1/12} {}_{2}F_{1}\bigg(\frac{1}{12}, \frac{5}{12}; 1; \frac{1}{J}\bigg),$$
where $J = J(\tau)$.
\end{theorem}

\begin{proof}
This follows from \cite[p. 255, (14)]{Archinard}, $\Delta(\tau) = \omega_1^{12} \Delta$, and the classical identity $\Delta(\tau) = (2\pi)^{12} \eta(\tau)^{24}$.
\end{proof}

The other three identities in \cite[p. 255, (14)-(15)]{Archinard} can be obtained from the above identity using Euler's and Pfaff's hypergeometric transformations.

\begin{theorem}[$J = 1$ case]
\label{thm:Jone}
Suppose $\tau = \omega_2/\omega_1$ is in the connected component $C_{(J - 1)/J, i}$ of the open set $|(J - 1)/J| < 1$. Then
$$\widetilde{\omega}_1 = \omega_1 \Delta^{1/12} = \frac{2\pi \alpha}{\tau + i} J^{-1/12} {}_{2}F_{1}\bigg(\frac{1}{12}, \frac{5}{12}; \frac{1}{2}; \frac{J - 1}{J}\bigg),$$
where $J = J(\tau)$ and $\alpha = 2i\eta(i)^2$.
\end{theorem}

\begin{proof}
This follows from \cite[p. 253, (10)]{Archinard}, $\Delta(\tau) = \omega_1^{12} \Delta$, and the classical identity $\Delta(\tau) = (2\pi)^{12} \eta(\tau)^{24}$.
\end{proof}

In \cite[p. 253, (10)-(11)]{Archinard}, there are five other identities convergent near $J = 1$. The ones involving $1 - J$ do not converge at the class number one singular values $\tau$. The remaining identity involving ${}_{2}F_{1}(7/12, 11/12; 3/2; (J - 1)/J)$ yields one further example of Chudnovsky-Ramanujan type formulae for level 1 using the method of this paper.

\begin{theorem}[$J = 0$ case]
\label{thm:Jzero}
Suppose $\tau = \omega_2/\omega_1$ is in the connected component $C_{J/(J - 1), \rho}$ of the open set $|J/(J - 1)| < 1$. Then
$$\widetilde{\omega}_1 = \omega_1 \Delta^{1/12} = \frac{2\pi \alpha}{\tau - \overline{\rho}} (1 - J)^{-1/12} {}_{2}F_{1}\bigg(\frac{1}{12}, \frac{7}{12}; \frac{2}{3}; \frac{J}{J - 1}\bigg),$$
where $J = J(\tau)$ and $\alpha = i \eta(\rho)^2 \sqrt{3}$. Here, $\rho = e^{2\pi i/3}$ is the cube root of unity.
\end{theorem}

\begin{proof}
This follows from \cite[p. 254, (12)]{Archinard}, $\Delta(\tau) = \omega_1^{12} \Delta$, and the classical identity $\Delta(\tau) = (2\pi)^{12} \eta(\tau)^{24}$.
\end{proof}

In \cite[pp. 254-255, (12)-(13)]{Archinard}, there are five other identities convergent near $J = 0$. The ones involving $J$ do not converge at the class number one singular values $\tau$. The remaining identity involving ${}_{2}F_{1}(5/12, 11/12; 4/3; J/(J - 1))$ yields one further example of Chudnovsky-Ramnujan type formulae for level 1 using the method of this paper.

\subsection{Quasi-periods}

For $z \in \mathbb{C}$, the Weierstrass $\zeta$-function is defined by
$$\zeta(z, \Lambda) = \frac{1}{z} + \sum_{\omega \in \Lambda'} \bigg[\frac{1}{z - \omega} + \frac{1}{\omega} + \frac{z}{\omega^2}\bigg],$$
where $\Lambda \subset \mathbb{C}$ is a lattice and $\Lambda' = \Lambda \setminus \{0\}$. That is, $\zeta'(z, \Lambda) = -\wp(z, \Lambda)$. The quasi-period $\eta(\omega, \Lambda)$ is defined by
$$\eta(\omega, \Lambda) = \zeta(z + \omega, \Lambda) - \zeta(z, \Lambda)$$
for $\omega \in \Lambda$. The dependence on $\Lambda$ is often suppressed if there is no confusion. Since
$$\zeta(\lambda z, \lambda \Lambda) = \lambda^{-1} \zeta(z, \Lambda),$$
(see, for instance, \cite[Chapter 18, Section 1]{Lang}) it follows that
$$\eta(\lambda \omega, \lambda \Lambda) = \lambda^{-1} \eta(\omega, \Lambda).$$

Define
$$\eta_k = \eta(\omega_k, \Lambda(E)), \quad \widetilde{\eta}_k = \eta(\widetilde{\omega}_k, \Lambda(\widetilde{E})), \quad {\rm H}_k = \eta(\Omega_k, \Lambda(E_J)).$$
From the homotheties relating the lattices $\Lambda(E)$, $\Lambda(\widetilde{E})$, $\Lambda(E_J)$, we have
\begin{align*}
\eta_k \Delta^{-1/12} &= \widetilde{\eta}_k, \\
\eta_k \sqrt{\frac{g_2}{g_3}} &= {\rm H}_k.
\end{align*}
Under the map of Riemann surfaces
$$\iota_{\Lambda(E_J)} : \mathbb{C} \to \mathbb{C} / \Lambda(E_J) \to E_{\Lambda(E_J)}(\mathbb{C}),$$
given by $\iota_{\Lambda(E_J)}(z) = [\wp(z, \Lambda(E_J)) : \wp'(z, \Lambda(E_J)) : 1]$, the differential $-xdx/y$ pulls back to the differential $d\zeta$ on $\mathbb{C}$. It follows that
\begin{align*}
{\rm H}_k &= \eta(\Omega_k, \Lambda(E_J)) \\
&= \zeta(z + \Omega_k, \Lambda(E_J)) - \zeta(z, \Lambda(E_J)) \\
&= \int_z^{z + \Omega_k} d\zeta(s, \Lambda(E_J)) \\
&= -\int_{\gamma_k(J)} \frac{x}{y} dx.
\end{align*}

Table 3 summarizes the relationships between the periods and quasi-periods of the elliptic curves $E$, $E_\tau$, $\widetilde{E}$, $E_J$.

\begin{table}[H]
\centering
\begin{tabular}{ccc}
\hline
$\text{Elliptic curve}$ & $\text{Periods}$ & $\text{Quasi-periods}$ \\
\hline
$E$ & $(\omega_1, \omega_2)$ & $(\eta_1, \eta_2)$ \\
$E_\tau$ & $(1, \tau)$ & $(\eta_1 \omega_1, \eta_2 \omega_1)$ \\
$\widetilde{E}$ & $(\widetilde{\omega}_1, \widetilde{\omega}_2) = (\omega_1 \Delta^{1/12}, \omega_2 \Delta^{1/12})$ & $(\widetilde{\eta}_1, \widetilde{\eta}_2) = (\eta_1 \Delta^{-1/12}, \eta_2 \Delta^{-1/12})$ \\
$E_J$ & $(\Omega_1, \Omega_2) = (\omega_1 \sqrt{g_3/g_2}, \omega_2 \sqrt{g_3/g_2})$ & $({\rm H}_1, {\rm H}_2) = (\eta_1 \sqrt{g_2/g_3}, \eta_2 \sqrt{g_2/g_3})$ \\
\hline
\end{tabular}
\caption{Summary: Periods and Quasi-periods}
\label{period-table}
\end{table}

\subsection{Quasi-period expressions}

The following result appears in \cite[(4.5)]{Chudnovsky2}, but we provide additional details of its proof.

\begin{theorem}
\label{thm:quasi}
For $k = 1, 2$, we have
$$\widetilde{\eta}_k = -2\sqrt{3} J^{2/3} \sqrt{J - 1} \frac{d\widetilde{\omega}_k}{dJ}.$$
\end{theorem}

\begin{proof}
Let $\omega = \omega_k$, $\eta = \eta_k$, $\Omega = \Omega_k$, ${\rm H} = {\rm H}_k$, with $k = 1, 2$. From \cite[p. 34]{Modulfunctionen}, we have the following differential relation:
$$36J(J - 1) \frac{d\Omega}{dJ} = 3(2 + J)\Omega - 2(J - 1){\rm H},$$
which gives
\begin{equation}
\label{eta-main}
{\rm H} = \frac{3}{2} \frac{(J + 2) \Omega - 12J(J - 1) \frac{d\Omega}{dJ}}{J - 1} = \frac{3}{2} \frac{J + 2}{J - 1} \Omega - 18J \frac{d\Omega}{dJ}.
\end{equation}
Recall that
$\Omega = \omega \sqrt{g_3/g_2}$, ${\rm H} = \eta \sqrt{g_2/g_3}$ and $\widetilde{\omega} = \omega \Delta^{1/12}$, $\widetilde{\eta} = \eta \Delta^{-1/12}$. Since
$$\Delta^{-1/12} \sqrt{\frac{g_3}{g_2}} = J^{-1/6} \bigg(\frac{27}{J - 1}\bigg)^{-1/4} \quad \text{and} \quad \Delta^{1/12} \sqrt{\frac{g_2}{g_3}} = J^{1/6} \bigg(\frac{27}{J - 1}\bigg)^{1/4},$$
it is plain that
$$\Omega = \widetilde{\omega} J^{-1/6} \bigg(\frac{27}{J - 1}\bigg)^{-1/4} \quad \text{and} \quad {\rm H} = \widetilde{\eta} J^{1/6} \bigg(\frac{27}{J - 1}\bigg)^{1/4},$$
and
$$\frac{d\Omega}{dJ} = \frac{3^{1/4} (J + 2)}{36J^{7/6} (J - 1)^{3/4}} \widetilde{\omega} + \frac{3^{1/4} (J - 1)^{1/4}}{3J^{1/6}} \frac{d\widetilde{\omega}}{dJ}.$$
Substituting these into \eqref{eta-main} gives
\begin{align*}
\widetilde{\eta} \bigg(\frac{27}{J - 1}\bigg)^{1/4} &= \frac{3}{2} \frac{J + 2}{J - 1} \widetilde{\omega} J^{-1/3} \bigg(\frac{27}{J - 1}\bigg)^{-1/4} \\
&\quad- 18J^{5/6} \bigg[\frac{3^{1/4} (J + 2)}{36J^{7/6} (J - 1)^{3/4}} \widetilde{\omega} + \frac{3^{1/4} (J - 1)^{1/4}}{3J^{1/6}} \frac{d\widetilde{\omega}}{dJ}\bigg] \\
&= -18J^{5/6} \frac{3^{1/4} (J - 1)^{1/4}}{3J^{1/6}} \frac{d\widetilde{\omega}}{dJ}.
\end{align*}
Solving for $\widetilde{\eta}$ we obtain
$$\widetilde{\eta} = -2\sqrt{3} J^{2/3} \sqrt{J - 1} \frac{d\widetilde{\omega}}{dJ}.$$
\end{proof}

\subsection{Complex multiplication period relations}
\label{CM}

Recall that $s_2(\tau)$ is defined by
$$s_2(\tau) = \frac{E_4(\tau)}{E_6(\tau)} \bigg(E_2(\tau) - \frac{3}{\pi {\rm Im}(\tau)}\bigg).$$

It is known that $s_2(\tau)$ is rational at $\tau = \sqrt{-N}$ for $N = 2, 3, 4, 7$ and at $\tau = \frac{-1 + \sqrt{-N}}{2}$ for $N = 7, 11, 19, 27, 43, 67, 163$ \cite[Lemma 4.1]{Chudnovsky2}. Below are tables giving these rational values.

\begin{table}[H]
\centering
\begin{tabular}{cc}
\hline
$N$ & $s_2(\tau)$ \\
\hline
$2$ & $5/14$ \\
$3$ & $5/11$ \\
$4$ & $11/21$ \\
$7$ & $85/133$ \\
\hline
\end{tabular}
\caption{Special values of $s_2(\tau)$ at $\tau = \sqrt{-N}$}
\end{table}

\begin{table}[H]
\centering
\begin{tabular}{cc}
\hline
$N$ & $s_2(\tau)$ \\
\hline
$7$ & $5/21$ \\
$11$ & $32/77$ \\
$19$ & $32/57$ \\
$27$ & $160/253$ \\
$43$ & $640/903$ \\
$67$ & $33440/43617$ \\
$163$ & $77265280/90856689$ \\
\hline
\end{tabular}
\caption{Special values of $s_2(\tau)$ at $\tau = \frac{-1 + \sqrt{-N}}{2}$}
\end{table}

From
$$g_2(\tau) = \frac{4\pi^4 E_4(\tau)}{3} \quad \text{and} \quad g_3(\tau) = \frac{8\pi^6 E_6(\tau)}{27},$$
and the fact that $g_2(\tau) = \omega_1^4 g_2$ and $g_3(\tau) = \omega_1^6 g_3$, we obtain:
\begin{equation}
\label{Ramanujan}
\frac{3g_3}{2g_2} s_2(\tau) = \frac{\pi^2}{3\omega_1^2} \bigg[E_2(\tau) - \frac{3}{\pi {\rm Im}(\tau)}\bigg].
\end{equation}

\begin{theorem}
We have
$$2i \omega_1 \eta_1 {\rm Im}(\tau) - \omega_1^2 \left[2i {\rm Im}(\tau) \frac{3g_3}{2g_2} s_2(\tau)\right] = 2\pi i.$$
\end{theorem}

\begin{proof}
From the evident identity
$$2\omega_1 \eta_1 {\rm Im}(\tau) - \bigg[2\omega_1 \eta_1 - \frac{2\pi}{{\rm Im}(\tau)}\bigg] {\rm Im}(\tau) = 2\pi,$$
we obtain
$$2\omega_1 \eta_1 {\rm Im}(\tau) - \frac{2\pi^2 {\rm Im}(\tau)}{3} \bigg[\frac{3\omega_1 \eta_1}{\pi^2} - \frac{3}{\pi {\rm Im}(\tau)}\bigg] = 2\pi.$$
From \cite[p. 298]{Greenhill} (or \cite[Chapter 18, Section 3]{Lang}), we have
$$E_2(\tau) = \frac{3\omega_1 \eta_1}{\pi^2}.$$
Hence
$$2\omega_1 \eta_1 {\rm Im}(\tau) - \frac{2\pi^2 {\rm Im}(\tau)}{3} \bigg[E_2(\tau) - \frac{3}{\pi {\rm Im}(\tau)} \bigg] = 2\pi.$$
Using \eqref{Ramanujan}, and multiplying through by $i$, completes the proof.
\end{proof}

\begin{remark}
If $\tau$ is as in \eqref{Chudnovsky}, then
$$\frac{\omega_1 \eta_1 \sqrt{-d}}{a} - \omega_1^2 \bigg[\frac{\sqrt{-d}}{a} \frac{3g_3}{2g_2} s_2(\tau)\bigg] = 2\pi i.$$
However, it will deem convenient to rewrite this as
\begin{equation}
\label{eq:Chudnovsky}
\frac{\omega_1}{2\pi} \bigg[\eta_1 - \omega_1 \frac{3g_3}{2g_2} s_2(\tau)\bigg] = \frac{a}{\sqrt{d}}.
\end{equation}
\end{remark}

\section{$J = \infty$ case}

In this section, we derive \eqref{Jinftypre} and thereby obtain \eqref{chudnovsky-formula-pi} and formulae like it.

\begin{lemma}
We have
\begin{equation}
\label{eq:Pochhammer}
\left(\frac{1}{6}\right)_n \left(\frac{5}{6}\right)_n \left(\frac{1}{2}\right)_n = 12^{-3n} \frac{(6n)!}{(3n)!}.
\end{equation}
\end{lemma}

\begin{proof}
Notice that
$$\left(\frac{p}{q}\right)_n = q^{-n} \prod_{k = 1}^n (qk + p - q), \quad p, q \in \mathbb{N}$$
follows directly from the definition of $(a)_n$. Therefore,
\begin{align*}
\left(\frac{1}{6}\right)_n \left(\frac{5}{6}\right)_n \left(\frac{3}{6}\right)_n &= 6^{-3n} \prod_{k = 1}^n (6k - 5)(6k - 3)(6k - 1) \\
&= \frac{3 \cdot 5 \cdot 7 \cdots (6n - 1)}{6^{3n}} \\
&= 6^{-3n} \frac{(6n)!}{2 \cdot 4 \cdot 6 \cdots 6n} \\
&= 12^{-3n} \frac{(6n)!}{(3n)!},
\end{align*}
and the proof is complete.
\end{proof}

The following is a slight generalization of what appears in \cite[(4.7)]{Chudnovsky2}.

\begin{theorem}
\label{thm:infty}
If $\tau$ is as in \eqref{Chudnovsky} and lies in $C_{1/J, \infty}$, then
$$\frac{a}{\pi \sqrt{d}} \frac{\sqrt{J}}{\sqrt{J - 1}} = F^2 \frac{1 - s_2(\tau)}{6} - J \frac{d}{dJ} F^2,$$
where $F = {}_{2}F_{1}(1/12, 5/12; 1; 1/J)$, $J = J(\tau)$, and the principal branch of the square root is used.
\end{theorem}

\begin{proof}
Let $F = {}_{2}F_{1}(1/12, 5/12; 1; 1/J)$. By Theorem~\ref{thm:Jinfty},
$$\widetilde{\omega}_1 = \frac{2\pi}{12^{1/4}} J^{-1/12} F \quad \text{and} \quad \omega_1 = \frac{2\pi}{12^{1/4}} (J\Delta)^{-1/12} F,$$
so
\begin{align*}
\frac{d \widetilde{\omega}_1}{dJ} &= \frac{2\pi}{12^{1/4}} \bigg(J^{-1/12} \frac{dF}{dJ} + F\frac{d}{dJ} J^{-1/12}\bigg) \\
&= \frac{2\pi}{12^{1/4}} \bigg(J^{-1/12} \frac{dF}{dJ} - \frac{F}{12J^{13/12}}\bigg) \\
&= -\frac{2\pi}{12^{1/4}} J^{-1/12} \bigg(\frac{F}{12J} - \frac{dF}{dJ}\bigg).
\end{align*}
Further, from Theorem~\ref{thm:quasi},
\begin{align*}
\eta_1 &= -2\sqrt{3} J^{2/3} \sqrt{J - 1} \frac{d \widetilde{\omega}_1}{dJ} \Delta^{1/12} \\
&= \frac{4\pi \sqrt{3}}{12^{1/4}} J^{7/12} \sqrt{J - 1} \bigg(\frac{F}{12J} - \frac{dF}{dJ}\bigg) \Delta^{1/12} \\
&= \frac{4\pi \sqrt{3}}{12^{1/4}} J^{-5/12} \sqrt{J - 1} \bigg(\frac{F}{12} - J\frac{dF}{dJ}\bigg) \Delta^{1/12}.
\end{align*}
Substituting the above expressions for $\omega_1$ and $\eta_1$ into \eqref{eq:Chudnovsky} gives
\begin{align*}
&\frac{(J \Delta)^{-1/12}}{12^{1/4}} F\bigg[\frac{2\sqrt{3}}{12^{1/4}} J^{-5/12} \sqrt{J - 1} \bigg(\frac{F}{12} - J\frac{dF}{dJ}\bigg) \Delta^{1/12} \\
&\quad- \frac{(J \Delta)^{-1/12}}{12^4} F\frac{3 g_3}{2 g_2} s_2(\tau)\bigg] = \frac{a}{2\pi \sqrt{d}},
\end{align*}
or, what is the same thing,
\begin{align*}
&\frac{\sqrt{J - 1}}{\sqrt{J}} F\bigg[\frac{F}{12} - J\frac{dF}{dJ} \\
&\quad- \frac{(J \Delta)^{-1/6}}{\sqrt{12}} F\frac{3 g_3}{2 g_2} \frac{\sqrt{J}}{\sqrt{J - 1}} s_2(\tau)\bigg] = \frac{a}{2\pi \sqrt{d}}.
\end{align*}
Using Lemma~\ref{trivial} leads to the desired result.
\end{proof}

\begin{remark}
To justify that the principal branch of the square root makes the formula valid, we numerically verified the formula at $\tau = \sqrt{-2}$, which establishes the identity (see the remark in Section \ref{identities}).
\end{remark}

The following is a slight generalization of what appears in \cite[(1.4)]{Chudnovsky}.

\begin{theorem}
If $\tau$ is as in \eqref{Chudnovsky} and lies in $C_{1/J, \infty}$, then
$$\frac{a}{\pi \sqrt{d}} \frac{\sqrt{j}}{\sqrt{j - 12^3}} = \sum_{n = 0}^\infty \left(\frac{1 - s_2(\tau)}{6} + n\right) \frac{(6n)!}{(3n)! n!^3} j^{-n}.$$
\end{theorem}

\begin{proof}
Using Clausen's formula \cite[p. 116]{Roy}
$${}_{2}F_{1}(a, b; c; z)^2 = {}_{3}F_{2}(2a, 2b, a + b; 2a + 2b, c; z), \quad c = a + b + \frac{1}{2},$$
we arrive at
$${}_{2}F_{1}\bigg(\frac{1}{12}, \frac{5}{12}; 1; \frac{1}{J}\bigg)^2 = {}_{3}F_{2}\bigg(\frac{1}{6}, \frac{5}{6}, \frac{1}{2}; 1, 1; \frac{1}{J}\bigg).$$
On account of \eqref{eq:Pochhammer} we have
$${}_{2}F_{1}\bigg(\frac{1}{12}, \frac{5}{12}; 1; z\bigg)^2 = \sum_{n = 0}^\infty \frac{(6n)!}{(3n)! n!^3} \frac{z^n}{12^{3n}},$$
hence
$$z\frac{d}{dz} {}_{2}F_{1}\bigg(\frac{1}{12}, \frac{5}{12}; 1; z\bigg)^2 = \sum_{n = 0}^\infty \frac{(6n)!}{(3n)! n!^3} \frac{n}{12^{3n}} z^n.$$
Changing $z$ into $1/z$ yields this
$$-z\frac{d}{dz} {}_{2}F_{1}\bigg(\frac{1}{12}, \frac{5}{12}; 1; \frac{1}{z}\bigg)^2 = \sum_{n = 0}^\infty \frac{(6n)!}{(3n)! n!^3} \frac{n}{12^{3n}} z^{-n}.$$
Upon setting $z = J = j/12^3$, and utilizing Theorem~\ref{thm:infty}, we obtain the desired result.
\end{proof}

\subsection{Examples}

The values $\tau = \sqrt{-N}$, where $N = 2, 3, 4, 7$, are such that $J(\tau)$ is rational, are as in \eqref{Chudnovsky}, and lie in $C_{1/J, \infty}$. Hence, the formula above holds whenever $\tau = \sqrt{-N}$, $N = 2, 3, 4, 7$. We state all the possible identities.

$\tau = \sqrt{-2}$:
$$\frac{5\sqrt{5}}{28\pi} = \sum_{n = 0}^\infty \bigg(\frac{3}{28} + n\bigg) \frac{(6n)!}{(3n)! n!^3} 20^{-3n}.$$

$\tau = \sqrt{-3}$:
$$\frac{5\sqrt{15}}{66\pi} = \sum_{n = 0}^\infty \bigg(\frac{1}{11} + n\bigg) \frac{(6n)!}{(3n)! n!^3} 2^{-n} 30^{-3n}.$$

$\tau = \sqrt{-4}$:
$$\frac{11\sqrt{33}}{252\pi} = \sum_{n = 0}^\infty \bigg(\frac{5}{63} + n\bigg) \frac{(6n)!}{(3n)! n!^3} 66^{-3n}.$$

$\tau = \sqrt{-7}$:
$$\frac{85\sqrt{255}}{7182\pi} = \sum_{n = 0}^\infty \bigg(\frac{8}{133} + n\bigg) \frac{(6n)!}{(3n)! n!^3} 225^{-3n}.$$

The values $\tau = \frac{-1 + \sqrt{-N}}{2}$, where $N = 7, 11, 19, 27, 43, 67, 163$, are such that $J(\tau)$ is rational, are as in \eqref{Chudnovsky}, and lie in $C_{1/J, \infty}$. So the theorem above holds for $\tau = \frac{-1 + \sqrt{-N}}{2}$, $N = 7, 11, 19, 27, 43, 67, 163$. We give all the possible formulae.

$\tau = \frac{-1 + \sqrt{-7}}{2}$:
$$\frac{5\sqrt{15}}{63\pi} = \sum_{n = 0}^\infty (-1)^n \bigg(\frac{8}{63} + n\bigg) \frac{(6n)!}{(3n)! n!^3} 15^{-3n}.$$

$\tau = \frac{-1 + \sqrt{-11}}{2}$:
$$\frac{16\sqrt{2}}{77\pi} = \sum_{n = 0}^\infty (-1)^n \bigg(\frac{15}{154} + n\bigg) \frac{(6n)!}{(3n)! n!^3} 32^{-3n}.$$

$\tau = \frac{-1 + \sqrt{-19}}{2}$:
$$\frac{16\sqrt{6}}{171\pi} = \sum_{n = 0}^\infty (-1)^n \bigg(\frac{25}{342} + n\bigg) \frac{(6n)!}{(3n)! n!^3} 96^{-3n}.$$

$\tau = \frac{-1 + \sqrt{-27}}{2}$:
$$\frac{80\sqrt{30}}{2277\pi} = \sum_{n = 0}^\infty (-1)^n \bigg(\frac{31}{506} + n\bigg) \frac{(6n)!}{(3n)! n!^3} 3^{-n} 160^{-3n}.$$

$\tau = \frac{-1 + \sqrt{-43}}{2}$:
$$\frac{320\sqrt{15}}{8127\pi} = \sum_{n = 0}^\infty (-1)^n \bigg(\frac{263}{5418} + n\bigg) \frac{(6n)!}{(3n)! n!^3} 960^{-3n}.$$

$\tau = \frac{-1 + \sqrt{-67}}{2}$:
$$\frac{880\sqrt{330}}{130851\pi} = \sum_{n = 0}^\infty (-1)^n \bigg(\frac{10177}{261702} + n\bigg) \frac{(6n)!}{(3n)! n!^3} 5280^{-3n}.$$

$\tau = \frac{-1 + \sqrt{-163}}{2}$:
$$\frac{213440\sqrt{10005}}{272570067\pi} = \sum_{n = 0}^\infty (-1)^n \bigg(\frac{13591409}{545140134} + n\bigg) \frac{(6n)!}{(3n)! n!^3} 640320^{-3n}.$$

\section{$J = 1$ case}

In this section we prove a theorem analogous to Theorem~\ref{thm:infty} which results from the hypergeometric representation of $\Omega_1$ in Theorem~\ref{thm:Jone}. We begin with the following proposition.

\begin{proposition}
\label{prop:Eisenstein}
We have
$$\frac{E_4(\tau)}{E_6(\tau)} = \frac{2\pi^2}{9\widetilde{\omega}_1^2} J^{1/3} \frac{\sqrt{27}}{\sqrt{J - 1}}, \quad \widetilde{\omega}_1 = \omega_1 \Delta^{1/12}.$$
\end{proposition}

\begin{proof}
Since
$$g_2(\tau) = \frac{4\pi^4 E_4(\tau)}{3} \quad \text{and} \quad g_3(\tau) = \frac{8\pi^6 E_6(\tau)}{27},$$
we get
$$\frac{E_4(\tau)}{E_6(\tau)} = \frac{2\pi^2}{9} \frac{g_2(\tau)}{g_3(\tau)}.$$
But $g_2(\tau) = \omega_1^4 g_2$ and $g_3(\tau) = \omega_1^6 g_3$, so
$$\frac{E_4(\tau)}{E_6(\tau)} = \frac{2\pi^2}{9\omega_1^2} \frac{g_2}{g_3}.$$
Lastly, it follows from Lemma~\ref{trivial} that
$$\frac{g_2}{g_3} = \frac{J^{1/3}}{\Delta^{1/6}} \frac{\sqrt{27}}{\sqrt{J - 1}},$$
so the proposed identity follows at once since $\widetilde{\omega}_1 = \omega_1 \Delta^{1/12}$.
\end{proof}

\begin{theorem}
\label{thm:intermediate1}
If $\tau$ is as in \eqref{Chudnovsky} and lies in $C_{(J - 1)/J, i}$, then
$$\frac{i}{\pi} \bigg[\frac{a(\tau + i)^2}{2i \alpha^2 \sqrt{3d}} \frac{\sqrt{J}}{\sqrt{J - 1}} - \frac{F^2}{\tau + i} \frac{E_4(\tau)}{E_6(\tau)}\bigg] = F^2 \frac{1 - s_2(\tau)}{6} - J\frac{d}{dJ} F^2,$$
where $F = {}_{2}F_{1}(1/12, 5/12; 1/2; (J - 1)/J)$ and $\alpha = 2i\eta(i)^2$.
\end{theorem}

\begin{proof}
Let $F = {}_{2}F_{1}(1/12, 5/12; 1/2; (J - 1)/J)$. Recall from Theorem~\ref{thm:Jone} that
$$\widetilde{\omega}_1 = J^{-1/12} \frac{m}{\tau_0} F,$$
where $m = 2\pi \alpha$ and $\tau_0 = \tau + i$. If $J' = \frac{dJ}{d\tau} \ne 0$, then locally $\tau$ is the inverse of $J$ and $\frac{d\tau}{dJ} = 1/J'$ by the inverse function theorem. Hence,
\begin{empheq}[left=\empheqlbrace]{align}
\label{inverse-assume}
\begin{split}
\frac{d \widetilde{\omega}_1}{dJ} &= J^{-1/12} \frac{m}{\tau_0} \frac{dF}{dJ} + F\frac{d}{dJ} \bigg(\frac{m}{\tau_0} J^{-1/12}\bigg) \\
&= J^{-1/12} \frac{m}{\tau_0} \bigg[\frac{dF}{dJ} - \bigg(\frac{1}{12J} + \frac{1}{\tau_0 J'}\bigg)F\bigg].
\end{split}
\end{empheq}
From Theorem~\ref{thm:quasi},
$$\eta_1 = -2\sqrt{3} J^{7/12} \sqrt{J - 1} \frac{m}{\tau_0} \bigg[\frac{dF}{dJ} - \bigg(\frac{1}{12J} + \frac{1}{\tau_0 J'}\bigg)F\bigg] \Delta^{1/12}.$$
Using \eqref{eq:Chudnovsky} we find that
\begin{align*}
\frac{a}{\sqrt{d}} &= -\frac{mF}{2\pi \tau_0 (J \Delta)^{1/12}} \bigg[2\sqrt{3} J^{7/12} \sqrt{J - 1} \frac{m}{\tau_0} \bigg[\frac{dF}{dJ} - F\bigg(\frac{1}{12J} + \frac{1}{\tau_0 J'}\bigg)\bigg] \Delta^{1/12} \\
&\qquad+ \frac{mF}{\tau_0 (J \Delta)^{1/12}} \frac{3g_3}{2g_2} s_2(\tau)\bigg]
\end{align*}
or
\begin{align*}
\frac{2\pi a \tau_0^2}{m^2 \sqrt{d}} &= -J^{-1/12} F\bigg[2\sqrt{3} J^{7/12} \sqrt{J - 1} \bigg[\frac{dF}{dJ} - F\bigg(\frac{1}{12J} + \frac{1}{\tau_0 J'}\bigg)\bigg] \\
&\qquad+ \frac{F}{J^{1/12} \Delta^{1/6}} \frac{3g_3}{2g_2} s_2(\tau)\bigg].
\end{align*}
Consequently,
\begin{align*}
\frac{2\pi a \tau_0^2}{m^2 \sqrt{d}} &= -J^{-1/12} F\bigg[2\sqrt{3} J^{7/12} \sqrt{J - 1} \bigg[\frac{dF}{dJ} - F\bigg(\frac{1}{12J} + \frac{1}{\tau_0 J'}\bigg)\bigg] \\
&\qquad+ \frac{F}{J^{1/12}} \frac{\sqrt{J - 1}}{2\sqrt{3} J^{1/3}} s_2(\tau)\bigg]
\end{align*}
because it is easy to see from Lemma~\ref{trivial} that
$$\frac{3g_3}{2g_2} \Delta^{-1/6} = \frac{\sqrt{J - 1}}{2\sqrt{3} J^{1/3}}.$$
Equivalently,
$$\frac{2\pi a \tau_0^2}{m^2 \sqrt{3d}} \frac{\sqrt{J}}{\sqrt{J - 1}} = F^2 \bigg[\frac{1}{6} + \frac{2J}{\tau_0 J'} - \frac{s_2(\tau)}{6}\bigg] - 2JF\frac{dF}{dJ}.$$
However,
$$2JF\frac{dF}{dJ} = J\frac{d}{dJ} F^2,$$
so
$$\frac{2\pi a \tau_0^2}{m^2 \sqrt{3d}} \frac{\sqrt{J}}{\sqrt{J - 1}} - F^2 \frac{2J}{\tau_0 J'} = F^2 \frac{1 - s_2(\tau)}{6} - J\frac{d}{dJ} F^2.$$
Hence, using \eqref{ramanujan-diff} we obtain
$$\frac{2\pi a \tau_0^2}{m^2 \sqrt{3d}} \frac{\sqrt{J}}{\sqrt{J - 1}} - \frac{i}{\pi} \frac{F^2}{\tau_0} \frac{E_4(\tau)}{E_6(\tau)} = F^2 \frac{1 - s_2(\tau)}{6} - J\frac{d}{dJ} F^2.$$
Since $m = 2\pi \alpha$, we have
$$\frac{i}{\pi} \bigg[\frac{a \tau_0^2}{2i \alpha^2 \sqrt{3d}} \frac{\sqrt{J}}{\sqrt{J - 1}} - \frac{F^2}{\tau_0} \frac{E_4(\tau)}{E_6(\tau)}\bigg] = F^2 \frac{1 - s_2(\tau)}{6} - J\frac{d}{dJ} F^2.$$
\end{proof}

\begin{remark}
The above derivation uses \eqref{inverse-assume}, which in turn is valid when $J' \ne 0$. This holds for $\tau \in C_{(J - 1)/J, i}$ aside for isolated points. Hence, the above identity holds for all $\tau \in C_{(J - 1)/J, i}$.
\end{remark}

\subsection{Proof of Theorem~\ref{thm:one}}

Using Proposition~\ref{prop:Eisenstein} and Theorem~\ref{thm:intermediate1} we obtain
$$\frac{i}{\pi} \bigg[\frac{a(\tau + i)^2}{2i \alpha^2 \sqrt{3d}} \frac{\sqrt{J}}{\sqrt{J - 1}} - \frac{F^2}{\tau + i} \frac{2\pi^2}{9\widetilde{\omega}_1^2} J^{1/3} \frac{\sqrt{27}}{\sqrt{J - 1}}\bigg] = F^2 \frac{1 - s_2(\tau)}{6} - J\frac{d}{dJ} F^2.$$
We see from Theorem~\ref{thm:Jone} that
$$\widetilde{\omega}_1^2 = \frac{4\pi^2 \alpha^2}{(\tau + i)^2 J^{1/6}} F^2,$$
so
$$\frac{i}{\pi} \bigg[\frac{a(\tau + i)^2}{2i \alpha^2 \sqrt{3d}} \frac{\sqrt{J}}{\sqrt{J - 1}} - \frac{\tau + i}{18\alpha^2} \frac{\sqrt{27J}}{\sqrt{J - 1}}\bigg] = F^2 \frac{1 - s_2(\tau)}{6} - J\frac{d}{dJ} F^2$$
or
$$\frac{i(\tau + i)}{\pi} \frac{\sqrt{J}}{\sqrt{J - 1}} \bigg(\frac{a(\tau + i)}{2i \alpha^2 \sqrt{3d}} - \frac{\sqrt{3}}{6\alpha^2}\bigg) = F^2 \frac{1 - s_2(\tau)}{6} - J\frac{d}{dJ} F^2,$$
which simplifies to
$$\frac{\tau + i}{2\pi \alpha^2 \sqrt{3}} \frac{\sqrt{J}}{\sqrt{1 - J}} \bigg(a\frac{\tau + i}{\sqrt{-d}} - 1\bigg) = F^2 \frac{1 - s_2(\tau)}{6} - J\frac{d}{dJ} F^2.$$

\begin{remark}
To justify that the principal branch of the square root makes the formula valid, we numerically verified the formula at $\tau = \sqrt{-2}$, which establishes the identity (see the remark in Section \ref{identities}).
\end{remark}

\subsection{Examples}
\label{one:examples}

From \eqref{hyperdiff}, we have
\begin{align*}
&\frac{d}{dJ} {}_{2}F_{1}\bigg(a, b; c; \frac{J - 1}{J}\bigg)^2 = \\
&\frac{2ab}{c J^2} {}_{2}F_{1}\bigg(a, b; c; \frac{J - 1}{J}\bigg) {}_{2}F_{1}\bigg(a + 1, b + 1; c + 1; \frac{J - 1}{J}\bigg).
\end{align*}

The values $\tau = \sqrt{-N}$, where $N = 2, 3, 4, 7$, are such that $J(\tau)$ is rational, are as in \eqref{Chudnovsky}, and lie in $C_{(J - 1)/J, i}$. Thus, the above theorem holds for $\tau = \sqrt{-N}$, $N = 2, 3, 4, 7$. Using the fact that
$$\eta(i) = \frac{\Gamma(1/4)}{2\pi^{3/4}},$$
that is, $\alpha = i \Gamma(1/4)^2/2\pi^{3/2}$, we give all the possible identities.

$\tau = \sqrt{-2}$:
\begin{align*}
&\frac{5\pi^2 (1 + \sqrt{2})(-2 + \sqrt{2}) \sqrt{2} \sqrt{3} \sqrt{5}}{84\Gamma(1/4)^4} = \\
&\frac{3}{28} {}_{2}F_{1}\bigg(\frac{1}{12}, \frac{5}{12}; \frac{1}{2}; \frac{98}{125}\bigg)^2 \\
&\quad- \frac{3}{100} {}_{2}F_{1}\bigg(\frac{1}{12}, \frac{5}{12}; \frac{1}{2}; \frac{98}{125}\bigg) {}_{2}F_{1}\bigg(\frac{13}{12}, \frac{17}{12}; \frac{3}{2}; \frac{98}{125}\bigg).
\end{align*}

$\tau = \sqrt{-3}$:
\begin{align*}
&\frac{5\pi^2 (1 + \sqrt{3})(-3 + \sqrt{3}) \sqrt{3} \sqrt{5}}{99\Gamma(1/4)^4} = \\
&\frac{1}{11} {}_{2}F_{1}\bigg(\frac{1}{12}, \frac{5}{12}; \frac{1}{2}; \frac{121}{125}\bigg)^2 \\
&\quad- \frac{1}{225} {}_{2}F_{1}\bigg(\frac{1}{12}, \frac{5}{12}; \frac{1}{2}; \frac{121}{125}\bigg) {}_{2}F_{1}\bigg(\frac{13}{12}, \frac{17}{12}; \frac{3}{2}; \frac{121}{125}\bigg).
\end{align*}

$\tau = \sqrt{-4}$:
\begin{align*}
&-\frac{11\pi^2 \sqrt{11}}{42\Gamma(1/4)^4} = \\
&\frac{5}{63} {}_{2}F_{1}\bigg(\frac{1}{12}, \frac{5}{12}; \frac{1}{2}; \frac{3^3 7^2}{11^3}\bigg)^2 \\
&\quad- \frac{10}{3^2 11^3} {}_{2}F_{1}\bigg(\frac{1}{12}, \frac{5}{12}; \frac{1}{2}; \frac{3^3 7^2}{11^3}\bigg) {}_{2}F_{1}\bigg(\frac{13}{12}, \frac{17}{12}; \frac{3}{2}; \frac{3^3 7^2}{11^3}\bigg).
\end{align*}

$\tau = \sqrt{-7}$:
\begin{align*}
&\frac{85\pi^2 (1 + \sqrt{7})(-7 + \sqrt{7}) \sqrt{7} \sqrt{85}}{3^3 7^2 19 \Gamma(1/4)^4} = \\
&\frac{8}{133} {}_{2}F_{1}\bigg(\frac{1}{12}, \frac{5}{12}; \frac{1}{2}; \frac{3^5 \cdot 7 \cdot 19^2}{5^3 17^3}\bigg)^2 \\
&\quad- \frac{16}{3^2 5^2 17^3} {}_{2}F_{1}\bigg(\frac{1}{12}, \frac{5}{12}; \frac{1}{2}; \frac{3^5 \cdot 7 \cdot 19^2}{5^3 17^3}\bigg) {}_{2}F_{1}\bigg(\frac{13}{12}, \frac{17}{12}; \frac{3}{2}; \frac{3^5 \cdot 7 \cdot 19^2}{5^3 17^3}\bigg).
\end{align*}

\section{$J = 0$ case}

In this section we prove a theorem analogous to Theorem~\ref{thm:infty} which results from the hypergeometric representation of $\Omega_1$ in Theorem~\ref{thm:Jzero}. Before we begin, we prove the following proposition which we use later.

\begin{proposition}
\label{Selberg}
Let $f(\tau)$ be the Weber's function and let $\chi(a) = (a \mid p)$ denote the Legendre symbol. If $p$ is an odd prime such that $\mathbb{Q}(\sqrt{-p})$ has class number $h(-p) = 1$, then
$$f(\sqrt{-p})^2 \eta(\sqrt{-p})^2 \sqrt{2\pi p} = \prod_{n = 1}^{p - 1} \Gamma\bigg(\frac{n}{p}\bigg)^{(w/4) \chi(n)},$$
and
$$\eta\left(\frac{-1 + \sqrt{-p}}{2}\right)^2 e^{\pi i/12} \sqrt{2\pi p} = \prod_{n = 1}^{p - 1} \Gamma\bigg(\frac{n}{p}\bigg)^{(w/4) \chi(n)},$$
where $w = 6$ if $p = 3$ and $w = 2$ if $p > 3$.
\end{proposition}

\begin{proof}
Let $K = K(k)$ be the complete elliptic integral of the first kind, where $0 < k < 1$, and let $K' = K(k')$, where $k' = \sqrt{1 - k^2}$.
Chowla and Selberg \cite[p. 89]{Chowla} proved that if $K'/K = \sqrt{p}$ and $h(-p) = 1$, then
\begin{equation}
\label{Chowla}
\frac{2K}{\pi} = 2^{1/3} \frac{(kk')^{-1/6}}{\sqrt{2\pi p}} \prod_{n = 1}^{p - 1} \Gamma\bigg(\frac{n}{p}\bigg)^{(w/4) \chi(n)},
\end{equation}
where $w = 6$ if $p = 3$ and $w = 2$ if $p > 3$.
Recall that $f(\tau)$ is defined for $\tau$ in the upper half-plane by \cite[p. 114]{Weber}
$$f(\tau) = q^{-1/48} \prod_{n = 1}^\infty \{1 + q^{(2n - 1)/2}\} = e^{-\pi i/24} \frac{\eta\left(\frac{1 + \tau}{2}\right)}{\eta(\tau)}, \quad q = e^{2\pi i \tau},$$
where $\eta(\tau)$ is the Dedekind eta function:
$$\eta(\tau) = q^{1/24} \prod_{n = 1}^\infty (1 - q^n).$$
Upon combining the following identity of Jacobi \cite[p. 481]{Watson}
$$\eta(\tau)^6 = q^{1/4} \prod_{n = 1}^\infty (1 - q^n)^6 = \frac{kk'}{4} \left(\frac{2K}{\pi}\right)^3, \quad q = e^{2 \pi i \tau},$$
and Weber's identity \cite[p. 179]{Weber}
$$f(\tau) = \frac{2^{1/6}}{(kk')^{1/12}},$$
we find that
\begin{equation}
\label{Weber}
f(\tau)^2 \eta(\tau)^2 = \frac{(kk')^{1/6}}{2^{1/3}} \frac{2K}{\pi}.
\end{equation}
Since we have the relation $K'/K = -(\log q)/2\pi$ \cite[p. 472]{Watson}, that is, $\tau = iK'/K$, taking $\tau = \sqrt{-p}$ shows that $K'/K = \sqrt{p}$. So from \eqref{Chowla} and \eqref{Weber} we get
$$f(\sqrt{-p})^2 \eta(\sqrt{-p})^2 \sqrt{2\pi p} = \prod_{n = 1}^{p - 1} \Gamma\bigg(\frac{n}{p}\bigg)^{(w/4) \chi(n)},$$
where $w = 6$ if $p = 3$ and $w = 2$ if $p > 3$. Moreover, it is classically known that $\eta(\tau + 1) = e^{\pi i/12} \eta(\tau)$, or, what is the same thing, $\eta(\tau) = e^{\pi i/12} \eta(\tau - 1)$, so
$$f(\tau) \eta(\tau) = e^{\pi i/24} \eta\left(\frac{-1 + \tau}{2}\right).$$
Therefore,
$$\eta\left(\frac{-1 + \sqrt{-p}}{2}\right)^2 e^{\pi i/12} \sqrt{2\pi p} = \prod_{n = 1}^{p - 1} \Gamma\bigg(\frac{n}{p}\bigg)^{(w/4) \chi(n)},$$
where $w = 6$ if $p = 3$ and $w = 2$ if $p > 3$.
\end{proof}

\begin{theorem}
\label{thm:intermediate0}
If $\tau$ is as in \eqref{Chudnovsky} and lies in $C_{J/(J-1),\rho}$, then
$$\frac{i}{\pi} \left[\frac{a(\tau - \overline{\rho})^2}{2\alpha^2 \sqrt{3d}} \frac{J^{1/3}}{(1 - J)^{1/3}} + \frac{F^2}{\tau - \overline{\rho}} \frac{E_4(\tau)}{E_6(\tau)}\right] = F^2 \bigg[\frac{J}{6(1 - J)} + \frac{s_2(\tau)}{6}\bigg] + J\frac{d}{dJ} F^2,$$
where $F = {}_{2}F_{1}(1/12, 7/12; 2/3; J/(J - 1))$ and $\alpha = i \eta(\rho)^2 \sqrt{3}$. Here, $\rho = e^{2\pi i/3}$ is the cube root of unity.
\end{theorem}

\begin{proof}
Let $F = {}_{2}F_{1}(1/12, 7/12; 2/3; J/(J - 1))$. Recall from Theorem~\ref{thm:Jzero} that
$$\widetilde{\omega}_1 = \frac{m}{\tau_0} (1 - J)^{-1/12} F,$$
where $m = 2\pi \alpha$ and $\tau_0 = \tau - \overline{\rho}$. If $J' = \frac{dJ}{d\tau} \ne 0$, then locally $\tau$ is the inverse of $J$ and $\frac{d\tau}{dJ} = 1/J'$ by the inverse function theorem. Thus,
\begin{align*}
\frac{d \widetilde{\omega}_1}{dJ} &= (1 - J)^{-1/12} \frac{m}{\tau_0} \frac{dF}{dJ} + F\frac{d}{dJ} \bigg(\frac{m}{\tau_0} (1 - J)^{-1/12}\bigg) \\
&= (1 - J)^{-1/12} \frac{m}{\tau_0} \bigg[\frac{dF}{dJ} + \bigg(\frac{1}{12(1 - J)} - \frac{1}{\tau_0 J'}\bigg)F\bigg].
\end{align*}
Using Theorem~\ref{thm:quasi} we get
$$\eta_1 = -\frac{2\sqrt{3} J^{2/3} \sqrt{J - 1}}{(1 - J)^{1/12}} \frac{m}{\tau_0} \bigg[\frac{dF}{dJ} + \bigg(\frac{1}{12(1 - J)} - \frac{1}{\tau_0 J'}\bigg)F\bigg] \Delta^{1/12}.$$
By virtue of \eqref{eq:Chudnovsky} we obtain
\begin{align*}
\frac{a}{\sqrt{d}} &= -\frac{mF}{2\pi \tau_0 ((1 - J) \Delta)^{1/12}} \bigg[\frac{2\sqrt{3} J^{2/3} \sqrt{J - 1}}{(1 - J)^{1/12}} \frac{m}{\tau_0} \bigg[\frac{dF}{dJ} + \bigg(\frac{1}{12(1 - J)} - \frac{1}{\tau_0 J'}\bigg)F\bigg] \Delta^{1/12} \\
&\qquad+ \frac{mF}{\tau_0 ((1 - J) \Delta)^{1/12}} \frac{3g_3}{2g_2} s_2(\tau)\bigg]
\end{align*}
or
\begin{align*}
\frac{2\pi a \tau_0^2}{m^2 \sqrt{d}} &= -\frac{F}{(1 - J)^{1/6}} \bigg[2\sqrt{3} J^{2/3} \sqrt{J - 1} \bigg[\frac{dF}{dJ} + \bigg(\frac{1}{12(1 - J)} - \frac{1}{\tau_0 J'}\bigg)F\bigg] \\
&\qquad+ \frac{F}{\Delta^{1/6}} \frac{3g_3}{2g_2} s_2(\tau)\bigg].
\end{align*}
Using Lemma~\ref{trivial} we get
$$\frac{3g_3}{2g_2} \Delta^{-1/6} = \frac{\sqrt{J - 1}}{2\sqrt{3} J^{1/3}},$$
which yields
\begin{align*}
\frac{2\pi a \tau_0^2}{m^2 \sqrt{d}} &= -\frac{F}{(1 - J)^{1/6}} \bigg[2\sqrt{3} J^{2/3} \sqrt{J - 1} \bigg[\frac{dF}{dJ} + \bigg(\frac{1}{12(1 - J)} - \frac{1}{\tau_0 J'}\bigg)F\bigg] \\
&\qquad+ \frac{\sqrt{J - 1}}{2\sqrt{3} J^{1/3}} s_2(\tau)F\bigg],
\end{align*}
or
$$\frac{2\pi a \tau_0^2}{m^2 \sqrt{3d}} \frac{J^{1/3}}{\sqrt{J - 1}} (1 - J)^{1/6} = -2JF\bigg[\frac{dF}{dJ} + \bigg(\frac{1}{12(1 - J)} - \frac{1}{\tau_0 J'}\bigg)F\bigg] - \frac{s_2(\tau)}{6} F^2.$$
Therefore,
$$-\frac{2\pi a \tau_0^2}{m^2 \sqrt{3d}} \frac{J^{1/3}}{\sqrt{J - 1}} (1 - J)^{1/6} + \frac{2J}{\tau_0 J'} F^2 = F^2 \bigg[\frac{J}{6(1 - J)} + \frac{s_2(\tau)}{6}\bigg] + J\frac{d}{dJ} F^2$$
as
$$2JF\frac{dF}{dJ} = J\frac{d}{dJ} F^2.$$
Using \eqref{ramanujan-diff} we obtain
$$-\frac{2\pi a \tau_0^2}{m^2 \sqrt{3d}} \frac{J^{1/3}}{\sqrt{J - 1}} (1 - J)^{1/6} + \frac{i}{\pi} \frac{F^2}{\tau_0} \frac{E_4(\tau)}{E_6(\tau)} = F^2 \bigg[\frac{J}{6(1 - J)} + \frac{s_2(\tau)}{6}\bigg] + J\frac{d}{dJ} F^2,$$
that is,
$$\frac{i}{\pi} \left[\frac{a \tau_0^2}{2\alpha^2 \sqrt{3d}} \frac{J^{1/3}}{(1 - J)^{1/3}} + \frac{F^2}{\tau_0} \frac{E_4(\tau)}{E_6(\tau)}\right] = F^2 \bigg[\frac{J}{6(1 - J)} + \frac{s_2(\tau)}{6}\bigg] + J\frac{d}{dJ} F^2.$$
\end{proof}

\subsection{Proof of Theorem~\ref{thm:zero}}

Using Proposition~\ref{prop:Eisenstein} and Theorem~\ref{thm:intermediate0} we obtain
$$\frac{i}{\pi} \bigg[\frac{a(\tau - \overline{\rho})^2}{2\alpha^2 \sqrt{3d}} \frac{J^{1/3}}{(1 - J)^{1/3}} + \frac{F^2}{\tau - \overline{\rho}} \frac{2\pi^2}{9\widetilde{\omega}_1^2} J^{1/3} \frac{\sqrt{27}}{\sqrt{J - 1}}\bigg] = F^2 \bigg[\frac{J}{6(1 - J)} + \frac{s_2(\tau)}{6}\bigg] + J\frac{d}{dJ} F^2.$$
We see from Theorem~\ref{thm:Jzero} that
$$\widetilde{\omega}_1^2 = \frac{4\pi^2 \alpha^2}{(\tau - \overline{\rho})^2 (1 - J)^{1/6}} F^2,$$
so
$$\frac{i}{\pi} \bigg(\frac{a(\tau - \overline{\rho})^2}{2\alpha^2 \sqrt{3d}} - \frac{\tau - \overline{\rho}}{6\alpha^2} \sqrt{-3}\bigg) \frac{J^{1/3}}{(1 - J)^{1/3}} = F^2 \bigg[\frac{J}{6(1 - J)} + \frac{s_2(\tau)}{6}\bigg] + J\frac{d}{dJ} F^2,$$
which simplifies to
$$-\frac{\tau - \overline{\rho}}{2\pi \alpha^2 \sqrt{3}} \frac{J^{1/3}}{(1 - J)^{1/3}} \bigg(a\frac{\tau - \overline{\rho}}{\sqrt{-d}} - 1\bigg) = F^2 \bigg[\frac{J}{6(1 - J)} + \frac{s_2(\tau)}{6}\bigg] + J\frac{d}{dJ} F^2.$$

\begin{remark}
To justify that the principal branch of the square and cube root makes the formula valid, we numerically verified the formula at $\tau = \frac{-1 + \sqrt{-7}}{2}$, which establishes the identity (see the remark in Section \ref{identities}).
\end{remark}

\subsection{Examples}
\label{zero:examples}

From \eqref{hyperdiff}, we get
\begin{align*}
&\frac{d}{dJ} {}_{2}F_{1}\bigg(a, b; c; \frac{J}{J - 1}\bigg)^2 = \\
&-\frac{2ab}{c(J - 1)^2} {}_{2}F_{1}\bigg(a, b; c; \frac{J}{J - 1}\bigg) {}_{2}F_{1}\bigg(a + 1, b + 1; c + 1; \frac{J}{J - 1}\bigg).
\end{align*}
Furthermore, setting $p = 3$ in Proposition \ref{Selberg} yields
$$\eta(\rho)^2 e^{\pi i/12} \sqrt{6\pi} = \frac{\Gamma(1/3)^{3/2}}{\Gamma(2/3)^{3/2}}.$$
It follows from Euler's reflection formula
$$\Gamma(z) \Gamma(1 - z) = \frac{\pi}{\sin \pi z}$$
with $z = 1/3$, that $\Gamma(1/3) \Gamma(2/3) \sqrt{3} = 2\pi$. Therefore:
$$\eta(\rho)^2 = \frac{3^{1/4} \Gamma(1/3)^3}{4\pi^2 e^{\pi i/12}}.$$

The values $\tau = \frac{-1 + \sqrt{-N}}{2}$, where $N = 7, 11, 19, 27, 43, 67, 163$, are such that $J(\tau)$ is rational, are as in \eqref{Chudnovsky} and lie in $C_{J/(J - 1), \rho}$. So the theorem above holds if $\tau = \frac{-1 + \sqrt{-N}}{2}$, $N = 7, 11, 19, 27, 43, 67, 163$. We state all the possible identities.

$\tau = \frac{-1 + \sqrt{-7}}{2}$:
\begin{align*}
&\frac{40\pi^3}{189\Gamma(1/3)^6} 7^{1/6} = \\
&-\frac{40}{567} {}_{2}F_{1}\bigg(\frac{1}{12}, \frac{7}{12}; \frac{2}{3}; \frac{125}{189}\bigg)^2 \\
&\quad+ \frac{500}{15309} {}_{2}F_{1}\bigg(\frac{1}{12}, \frac{7}{12}; \frac{2}{3}; \frac{125}{189}\bigg) \\
&\qquad \times {}_{2}F_{1}\bigg(\frac{13}{12}, \frac{19}{12}; \frac{5}{3}; \frac{125}{189}\bigg).
\end{align*}

$\tau = \frac{-1 + \sqrt{-11}}{2}$:
\begin{align*}
&\frac{128\pi^3}{693\Gamma(1/3)^6} 11^{1/6} 7^{1/3} = \\
&-\frac{48}{539} {}_{2}F_{1}\bigg(\frac{1}{12}, \frac{7}{12}; \frac{2}{3}; \frac{512}{539}\bigg)^2 \\
&\quad+ \frac{288}{41503} {}_{2}F_{1}\bigg(\frac{1}{12}, \frac{7}{12}; \frac{2}{3}; \frac{512}{539}\bigg) \\
&\qquad \times {}_{2}F_{1}\bigg(\frac{13}{12}, \frac{19}{12}; \frac{5}{3}; \frac{512}{539}\bigg).
\end{align*}

$\tau = \frac{-1 + \sqrt{-19}}{2}$:
\begin{align*}
&\frac{256\pi^3}{513\Gamma(1/3)^6} 19^{1/6} = \\
&-\frac{112}{1539} {}_{2}F_{1}\bigg(\frac{1}{12}, \frac{7}{12}; \frac{2}{3}; \frac{512}{513}\bigg)^2 \\
&\quad+ \frac{224}{789507} {}_{2}F_{1}\bigg(\frac{1}{12}, \frac{7}{12}; \frac{2}{3}; \frac{512}{513}\bigg) \\
&\qquad \times {}_{2}F_{1}\bigg(\frac{13}{12}, \frac{19}{12}; \frac{5}{3}; \frac{512}{513}\bigg).
\end{align*}

$\tau = \frac{-1 + \sqrt{-27}}{2}$:
\begin{align*}
&\frac{640\pi^3}{6831 \Gamma(1/3)^6} 27^{1/6} 253^{1/3} = \\
&-\frac{3920}{64009} {}_{2}F_{1}\bigg(\frac{1}{12}, \frac{7}{12}; \frac{2}{3}; \frac{64000}{64009}\bigg)^2 \\
&\quad+ \frac{84000}{4097152081} {}_{2}F_{1}\bigg(\frac{1}{12}, \frac{7}{12}; \frac{2}{3}; \frac{64000}{64009}\bigg) \\
&\qquad \times {}_{2}F_{1}\bigg(\frac{13}{12}, \frac{19}{12}; \frac{5}{3}; \frac{64000}{64009}\bigg).
\end{align*}

$\tau = \frac{-1 + \sqrt{-43}}{2}$:
\begin{align*}
&\frac{6400\pi^3}{24381\Gamma(1/3)^6} 43^{1/6} 21^{1/3} = \\
&-\frac{74560}{1536003} {}_{2}F_{1}\bigg(\frac{1}{12}, \frac{7}{12}; \frac{2}{3}; \frac{512000}{512001}\bigg)^2 \\
&\quad+ \frac{32000}{112347867429} {}_{2}F_{1}\bigg(\frac{1}{12}, \frac{7}{12}; \frac{2}{3}; \frac{512000}{512000}\bigg) \\
&\qquad \times {}_{2}F_{1}\bigg(\frac{13}{12}, \frac{19}{12}; \frac{5}{3}; \frac{512000}{512001}\bigg).
\end{align*}

$\tau = \frac{-1 + \sqrt{-67}}{2}$:
\begin{align*}
&\frac{56320\pi^3}{392553\Gamma(1/3)^6} 67^{1/6} 217^{1/3} = \\
&-\frac{9937840}{255552003} {}_{2}F_{1}\bigg(\frac{1}{12}, \frac{7}{12}; \frac{2}{3}; \frac{85184000}{85184001}\bigg)^2 \\
&\quad+ \frac{5324000}{3109848868443429} {}_{2}F_{1}\bigg(\frac{1}{12}, \frac{7}{12}; \frac{2}{3}; \frac{85184000}{85184001}\bigg) \\
&\qquad \times {}_{2}F_{1}\bigg(\frac{13}{12}, \frac{19}{12}; \frac{5}{3}; \frac{85184000}{85184001}\bigg).
\end{align*}

$\tau = \frac{-1 + \sqrt{-163}}{2}$:
\begin{align*}
&\frac{17075200\pi^3}{817710201\Gamma(1/3)^6} 163^{1/6} 185801^{1/3} = \\
&-\frac{11363838226240}{455794119168003} {}_{2}F_{1}\bigg(\frac{1}{12}, \frac{7}{12}; \frac{2}{3}; \frac{151931373056000}{151931373056001}\bigg)^2 \\
&\quad+ \frac{9495710816000}{9892775193720748560806619429} {}_{2}F_{1}\bigg(\frac{1}{12}, \frac{7}{12}; \frac{2}{3}; \frac{151931373056000}{151931373056001}\bigg) \\
&\qquad \times {}_{2}F_{1}\bigg(\frac{13}{12}, \frac{19}{12}; \frac{5}{3}; \frac{151931373056000}{151931373056001}\bigg).
\end{align*}

\section{Further work}

It is natural to apply the method of this paper to systematically derive Chudnovsky-Ramanajan type formulae for other families of elliptic curves, which we hope to do in our future work. For the interested reader, it is perhaps instructive to briefly discuss another example to give a sense of the generality of this method.

Consider the Legendre family of elliptic curves given by $y^2 = x(x - 1)(x - \lambda)$. The Picard-Fuchs differential equation for this family is well known and given by
$$\lambda(1 - \lambda) \frac{d^2 \Omega}{d\lambda^2} + (1 - 2\lambda) \frac{d\Omega}{d\lambda} - \frac{\Omega}{4} = 0,$$
which is a hypergeometric differential equation with parameters $a = 1/2$, $b = 1/2$, $c = 1$, and it has three regular singular points: $0, 1, \infty$.

Kummer's method yields six (distinct) hypergeometric solutions of the form
$$\lambda^\alpha (1 - \lambda)^\beta {}_{2}F_{1}\bigg(\frac{1}{2}, \frac{1}{2}; 1; \nu(\lambda)\bigg),$$
where $\nu(\lambda)$ is one of
$$\lambda, \quad 1 - \lambda, \quad \frac{1}{\lambda}, \quad \frac{1}{1 - \lambda}, \quad \frac{\lambda}{\lambda - 1}, \quad \frac{\lambda - 1}{\lambda}.$$
In particular, $\alpha = 0$ and $\beta = 0$ if $\nu(\lambda) = \lambda, 1 - \lambda$; $\alpha = -1/2$ and $\beta = 0$ if $\nu(\lambda) = 1/\lambda, (\lambda - 1)/\lambda$; $\alpha = 0$ and $\beta = -1/2$ if $\nu(\lambda) = 1/(1 - \lambda), \lambda/(\lambda - 1)$. Each of these solutions will be valid near one of the singular points $0, 1, \infty$ and will give rise to a hypergeometric representation of $\Omega_1$ in terms of $\lambda$. Applying the method used in this paper with $\lambda$ in place of $J$, one can derive Chudnovsky-Ramanujan type formulae corresponding to each hypergeometric representation of $\Omega_1$, which will be valid near one of the cusps $0, 1, \infty$.

In fact, according to \cite{Chudnovsky2}, some of Ramanujan's original formulae in \cite{Ramanujan2} are derived using the hypergeometric representations of the periods of the Legendre family. So this case was considered earlier than the level 1 case studied by D. V. and G. V. Chudnovsky. It would be interesting to do a complete determination using the method of this paper.

\end{document}